\documentclass[10.9pt,a4paper, reqno]{amsart}
\usepackage{a4wide}

\usepackage[utf8]{inputenc}
\usepackage[english]{babel}

\usepackage{amsfonts,amssymb,amsmath}

\usepackage{xcolor}
\usepackage[colorlinks,
    linkcolor={red!50!black},
    citecolor={blue!50!black},
    urlcolor={blue!80!black}]{hyperref}

\usepackage{tikz}
\usepackage[all]{xy}
\usepackage{enumitem}
\makeatletter
\newcommand{\mylabel}[2]{#2\def\@currentlabel{#2}\label{#1}}
\usetikzlibrary{calc, matrix, arrows, cd}

\newcommand{\bsm}{\left(\begin{smallmatrix}}
\newcommand{\esm}{\end{smallmatrix}\right)}
\newtheorem{innercustomthm}{Theorem}
\newenvironment{customthm}[1]
  {\renewcommand\theinnercustomthm{#1}\innercustomthm}
  {\endinnercustomthm}

\newtheorem{theorem}{Theorem}[section]

\newtheorem{lemma}[theorem]{Lemma}
\newtheorem{proposition}[theorem]{Proposition}

\theoremstyle{definition}
\newtheorem{definition}[theorem]{Definition}

\newtheorem{example}[theorem]{Example}

\newtheorem{remark}[theorem]{Remark}

\newtheorem{notation}[theorem]{Notation}
\newtheorem{convention}[theorem]{Convention}
\newtheorem{construction}[theorem]{Construction}
\newtheorem{claim}{Claim}
\newtheorem*{claim*}{Claim}

\newcommand{\wt}{\widetilde}

\newcommand{\sm}{\setminus}

\newcommand{\im}{\operatorname{im}}
\newcommand{\coker}{\operatorname{coker}}

\newcommand{\proj}{\operatorname{proj}}

\newcommand{\rad}{\operatorname{rad}}

\newcommand{\divv}{\operatorname{div}}

 \definecolor{bettergreen}{rgb}{0.0, 0.5 0.0}

\newcommand{\Z}{\mathbb{Z}}
\newcommand{\Q}{\mathbb{Q}}
\newcommand{\R}{\mathbb{R}}
\newcommand{\C}{\mathbb{C}}

\newcommand{\ks}{\operatorname{ks}}
\newcommand{\red}{\operatorname{red}}
\newcommand{\id}{\operatorname{id}}

\newcommand{\Aut}{\operatorname{Aut}}

\title{Non-smoothable surfaces in the 4-sphere}

\author{Anthony Conway}
\author{Daniel Galvin}
\address{The University of Texas at Austin, Austin TX 78712}
\email{anthony.conway@austin.utexas.edu}
\email{daniel.galvin@austin.utexas.edu}

\begin{document}
\begin{abstract}
We construct examples of non-smoothable surfaces in the $4$-sphere, thereby answering Question~4.32 on the K3 problem list.
These surfaces are non-orientable and have knot group of order~$2$,  thus simultaneously answering Question~4.29(a) on the K3 problem list.
\end{abstract}
\maketitle

\section{Introduction}
It has been known since the work of Freedman~\cite{Freedman} and Donaldson~\cite{Donaldson} that~$4$-manifolds exhibit radically different behaviors in the smooth and topological categories.
To name but two examples, there are smooth~$4$-manifolds that are homeomorphic but not diffeomorphic, 
and there are topological manifolds that do not admit any smooth structure.
The difference in categories is also witnessed by knotted surfaces.
Examples of \emph{exotic surfaces} (i.e.  surfaces that are topologically isotopic but not smoothly isotopic) were first discovered by Finashin-Kreck-Viro~\cite{FinashinKreckViro} whereas examples of \emph{non-smoothable} surfaces (i.e. locally flat surfaces that are not isotopic to any smoothly embedded surface) first appeared in~\cite{Kuga,Rudolph}; see also~\cite{Luo,LeeWilczyGenus,Torres}.

The difficulty in finding differences between the smooth and topological categories typically increases as the amount of ambient algebraic topology decreases.
The most challenging questions in the area therefore generally involve the~$4$-sphere~$S^4$.
While the best known example is the smooth~$4$-dimensional Poincar\'e conjecture, 
there are also the questions of whether there exist exotic oriented surfaces in~$S^4$ and  whether there exist non-smoothable surfaces in~$S^4$.
Our main theorem addresses the latter, which is Question 4.32 on the K3-problem list~\cite{K3}.

\begin{theorem}
\label{thm:Main}
For every $h \geq 9$,  there is a non-smoothable surface of non-orientable genus $h$ in~$S^4$.
The fundamental group of the complement of each of these surfaces has order $2$.
\end{theorem}

We describe how Theorem~\ref{thm:Main} also answers Question 4.29(a) on the K3-problem list.
First,  some terminology is in order.
A non-orientable surface in~$S^4$ is \emph{unknotted} if it is isotopic to a connected
sum of unknotted projective planes.
In turn, a projective plane in~ $S^4$ is \emph{unknotted} if it is isotopic to the branch locus of one of the branched covers~$\C P^2 \to S^4$ or~$\overline{\C P}^2 \to S^4$ induced by complex conjugation.

The fundamental group of the complement, also called the \emph{knot group},  of such non-orientable unknotted surfaces has order $2$.
Whether or not the converse holds,  i.e.\ whether \emph{$\Z_2$-surfaces} are necessarily unknotted, has been asked in Question 4.29(a) on the K3-problem list as well as in~e.g.~\cite[p.~181]{KawauchiSurvey},~\cite[p.~55]{KamadaBook}, ~\cite[p.~7]{Matrix} and~\cite[Question~1.1]{ConwayOrsonPowell}.
By virtue of being non-smoothable, the $\Z_2$-surfaces in Theorem~\ref{thm:Main} are necessarily knotted, thus answering this question in the negative.

\subsection{Further context and comparisons}
A driving question in the study of knotted surfaces asks whether an embedded surface in $S^4$ with cyclic knot group is necessarily unknotted.
Here, note the knot group is $\Z$ or $\Z_2$ according to whether or not the surface is orientable.
In the topological category,  the answer is positive for orientable surfaces of genus $g \neq 1,2$~\cite{FreedmanQuinn,ConwayPowell} (and remains open for $g=1,2$) as well as for surfaces with non-orientable genus~$h$ and normal Euler number~$e$, when~$|e| \neq 2h$~\cite[Theorem A]{ConwayOrsonPowell}.
The same holds for~$h \leq 5$ and~$|e|=2h$ (with the case $h=1$ originally due to Lawson~\cite{Lawson} and the cases~$h=4,5$ requiring~\cite{Pencovitch}); see also~\cite{FinashinKreckViro,KreckOnTheHomeomorphism} for early work on the topic.
Combining these results with Theorem~\ref{thm:Main} addresses the question of whether~$\Z_2$-surfaces of non-orientable genus~$h$ are unknotted for every~$h \neq 6,7,8$.
In fact, in the topological category,  the aforementioned question was often phrased as the conjecture that locally flat surfaces in $S^4$ with cyclic knot group are unknotted, see e.g.~\cite{Matrix,K3}.  
In the smooth category,  the question remains open in the orientable case, but admits a negative answer in the non-orientable case~\cite{FinashinKreckViro,Finashin,Miyazawa,MaticOzturkReyesStipsiczUrzua}.

\subsection{Outline of the proof}
There are natural candidates for non-smoothable~$\Z_2$-surfaces, namely those whose double branched cover have a definite but non-diagonalisable form.
The challenge is to construct a surface realising such a form.
Section~\ref{sec:ExteriorNonDegenerateForm} associates to every odd non-singular symmetric bilinear form~$b \colon \Z^h \times \Z^h \to \Z$ a non-degenerate form~$b^{nd}_{ext}$ that is the algebraic incarnation of the intersection form on the universal cover of the surface exterior; we call this form the \emph{exterior non-degenerate form.}
To this form Section~\ref{sec:Pullback} associates a $\Z[\Z_2]$-valued hermitian form $\lambda_b$ we call the \emph{exterior equivariant form} via a pullback construction due to Hambleton-Riehm~\cite{HambletonRiehm}:
$$
\xymatrix@R0.4cm@C0.3cm{
(\Z_- \oplus \Lambda^{h-1},\lambda_b) \oplus H(\Lambda)^{\oplus 2} \ar[r]\ar[d]& (\Z_-^h,2b_{ext}^{nd})\ar[d] \oplus H(\Z_+)^{\oplus 2} \\
(\Z_+^{h-1},0)\ar[r]\oplus H(\Z_+)^{\oplus 2}& (\Z_2^{h-1},0) \oplus H(\Z_2)^{\oplus 2}.
}
$$
Here,  we write~$\Lambda:=\Z[\Z_2]$ and~$\Z_{\pm}:=\Lambda/(T\mp~1)$ as well as~$H(\Lambda),H(\Z_\pm),H(\Z_2)$ for the corresponding hyperbolic forms.
The maps in this diagram will be described in  Construction~\ref{conv:Identification}.

We observe in Proposition~\ref{prop:StablyRealisable},  using the work of Freedman~\cite{Freedman, PowellRayTeichner},
that this latter form is realisable as the equivariant intersection form of a surface exterior if and only if it is stably realisable, that is realisable after taking the direct sum with some number of hyperbolic forms.

 Section~\ref{sec:Proof} studies the stable realisability of~$\lambda_b$ by using properties of the pullback.
 Specifically, we compare $\lambda_b$ to the equivariant intersection form $\lambda_u$ of the exterior of the Euler number~$-2h$ nonorientable genus $h$ unknotted surface~$U \subset S^4$ which also fits in a pullback square of the form
 $$
\xymatrix@R0.4cm@C0.3cm{
(\Z_- \oplus \Lambda^{h-1},\lambda_u) \oplus H(\Lambda)^{\oplus 2} \ar[r]\ar[d]& (\Z_-^h,2u_{ext}^{nd})\ar[d] \oplus H(\Z)^{\oplus 2} \\
(\Z_+^{h-1},0)\ar[r]\oplus H(\Z_+)^{\oplus 2}& (\Z_2^{h-1},0) \oplus H(\Z_2)^{\oplus 2}. 
}
$$ 
Here, the diagonal form~$u:=(1)^{\oplus h}$ is the intersection form of the $2$-fold branched cover $\Sigma_2(U)$.

In order to stably realise~$\lambda_b$,  Proposition~\ref{prop:alpha-} first uses the isometry $b \oplus H(\Z) \cong u \oplus H(\Z)$ to build an isometry
$$\beta \colon (\Z_-^h,2u^{nd}_{ext})\oplus H(\Z_-)^{\oplus 2}\to(\Z_-^h,2b^{nd}_{ext})\oplus H(\Z_-)^{\oplus 2} .
$$
When $b \cong b' \oplus (1)$ with $b'$ an even definite non-singular symmetric bilinear form,  Proposition~\ref{prop:alpha-Fix} then shows that $\beta$ can be modified so as to yield a new isometry $\alpha_-$ that descends to an isometry
$$
\alpha_2 \colon (\Z_2^{h-1},0)\oplus H(\Z_2)^{\oplus 2} \to (\Z_2^{h-1},0)\oplus H(\Z_2)^{\oplus 2}.
$$
Proposition~\ref{prop:alpha+} then lifts $\alpha_2$ to an isometry
$$
\alpha_+ \colon (\Z_+^{h-1},0)\oplus H(\Z_+)^{\oplus 2} \to (\Z_+^{h-1},0)\oplus H(\Z_+)^{\oplus 2}.
$$
The stable realisability of~$\lambda_b$ then essentially follows from the definition of pullbacks.
The outcome is a~$\Z_2$-surface~$F \subset S^4$ for which the (non-degenerate part of the) intersection form on the universal cover of the surface exterior is~$b^{nd}_{ext}$.
Finally,  using  results in lattice theory, we show that the branched double cover~$\Sigma_2(F)$ has intersection form~$Q_{\Sigma_2(F)}$ isometric to~$b$.

In summary, we will prove the following theorem.

\begin{theorem}\label{thm:precise}
If~$b'\colon \Z^{h-1} \times \Z^{h-1} \to \Z$ is an even definite non-singular symmetric bilinear form, then there exists a non-orientable~$\Z_2$-surface~$F\subset S^4$ such that~$Q_{\Sigma_2(F)} \cong b:=b'\oplus(1)$.
\end{theorem}

We explain how Theorem~\ref{thm:precise} combines with work of Donaldson~\cite{Donaldson} to yield Theorem~\ref{thm:Main}.

\begin{proof}[Proof of Theorem \ref{thm:Main}]
Set~$b'=E_8$ and apply Theorem~\ref{thm:precise} to construct a~$\Z_2$-surface~$F\subset S^4$ such that~$Q_{\Sigma_2(F)} \cong b:=b'\oplus (1)$.
Now let~$k\geq 0$ and let~$U_k\subset S^4$ be the unknotted surface of non-orientable genus~$k$ with Euler number~$-2k$ (if~$k=0$ this is by convention the unknotted~$2$-sphere).  
Form the surface~$F_k:=F\# U_k$.  
Since~$\Sigma_2(U_k)\cong \#_k \C P^2$, we deduce the homeomorphism~$\Sigma_2(F_k)\cong \Sigma_2(F)\# (\#_k \C P^2)$ and hence~$Q_{\Sigma_2(F_k)}\cong b\oplus (1)^{\oplus k}$.  
These forms are non-diagonalisable, and hence~$F_k$ is a knotted surface of non-orientable genus~$9+k$. 
By Donaldson~\cite{Donaldson},~$Q_{\Sigma_2(F_k)}$ is not realised by a smooth~$4$-manifold and so~$\Sigma_2(F)$ is non-smoothable.  
We deduce that~$F_k \subset S^4$ is non-smoothable.
\end{proof}
\begin{remark}
Realising the form~$E_8\oplus(1)$ suffices to prove Theorem \ref{thm:Main} but we nevertheless wish to emphasise that Theorem~\ref{thm:precise} leads to the following stronger statement: ``For every~$h \geq~9$,  with $h\equiv 1\pmod{8}$,  there are~$N(h)$ non-smoothable~$\Z_2$-surfaces of nonorientable genus~$h$ in~$S^4$, where~$N(h)$ denotes the number of non-diagonalisable even definite non-singular symmetric bilinear forms of rank~$h-1$."  
The Minkowski-Siegel mass formula implies that~$N(h)$ likely grows super-exponentially; we refer to~\cite[pages~49, 403 and Chapter~16]{ConwaySloane} for a detailed discussion.
For example,  there are at least~$10^{9}$ non-smoothable~$\Z_2$-surfaces of non-orientable genus~$33$, pairwise distinct up to isotopy.
\end{remark}

\begin{remark}
The branched double covers~$\Sigma_2(F)$ for the surfaces in the above theorems have trivial Kirby-Siebenmann invariant; this follows from the fact that if a closed oriented~$4$-manifold~$X$ supports an orientation preserving locally linear involution, then~$\ks(X)=0$~\cite{KwasikVogel}; see also~\cite[Remark on page 120]{EdmondsAspects}.
  Accordingly, in the simplest case of our theorem (for~$b=E_8\oplus (1)$) we have that~$\Sigma_2(F) \cong E_8\# \ast \C P^2 $, where~$E_8$ denotes Freedman's~$E_8$ manifold and~$\ast \C P^2$ denotes the Chern manifold, which is homotopy equivalent but not homeomorphic to~$\C P^2$.
This surface~$F\subset S^4$ has non-orientable genus $9$ and normal Euler number $-18$.
\end{remark}

\subsection*{Organisation}

Section~\ref{sec:StableRealisation} recalls some facts about intersection forms and notes that the realisation problem is a stable question.
Section~\ref{sec:ExteriorNonDegenerateForm} introduces the exterior non-degenerate form.
Section~\ref{sec:Pullback} recalls Hambleton-Riehm's pullback construction of hermitian forms.
Section~\ref{sec:Proof} proves Theorems~\ref{thm:precise}.
Appendix~\ref{sec:lattices} contains a brief discussion of lattices and Kneser neighbours.

\subsection*{Acknowledgments}
AC was partially supported by the NSF grant DMS~2303674.
We are very grateful to Peter Teichner for pointing out a gap in a previous version of this work.
We also thank Lisa Piccirillo for helping us with the handle diagrams involved in the proof of Proposition~\ref{prop:IdentificationUnknot}.

\subsection*{Conventions}
We work in the topological category with locally flat embeddings.
Manifolds are assumed to be compact, connected and oriented unless otherwise specified.
Modules are assumed to be finitely generated.  
	To simplify notation the proofs are written with all definite forms considered to be positive.  The interested reader can readily extend all herein arguments to negative definite forms by substituting the non-orientable genus $h$ unknot with Euler number $-2h$ for the unknot with Euler number $+2h$ and making the corresponding subsequent modifications.

\section{Realisation and stable realisation of hermitian forms}
\label{sec:StableRealisation}

Given a finite group $\pi$, this section is concerned with the problem of realising a~$\Z[\pi]$-valued hermitian form as the equivariant intersection form of a~$4$-manifold with a prescribed boundary and fundamental group $\pi$.
After introducing some terminology,  we show that this is essentially a stable problem (Proposition~\ref{prop:StablyRealisable}).

\medbreak

Fix a closed~$3$-manifold~$Y$ and an epimorphism~$\varphi \colon \pi_1(Y) \twoheadrightarrow \pi$ to a finite group~$\pi$.
We say that a~$4$-manifold~$X$ with~$\pi_1(X) \cong \pi$ is a \emph{filling of~$(Y,\varphi)$} if there is a homeomorphism~$ \partial X \cong Y$ such that the composition~$\pi_1(Y) \cong \pi_1(\partial X) \to \pi_1(X) \cong \pi$ agrees with~$\varphi$.
In what follows, given a filling~$X$ of~$(Y,\varphi)$ with~$\pi_1(X) \cong \pi$, we write~$\widetilde{X}$ for its universal cover,~$Y^\varphi=\partial \widetilde{X}$ for its boundary,~$Q_{\widetilde{X}}$ for the~$\Z$-intersection form on~$H_2(\widetilde{X})$, and~$Q_{\widetilde{X}}^{nd}$ for the induced non-degenerate form on~$\coker(i_* \colon H_2(Y^{\varphi}) \to H_2(\widetilde{X}))$.
Here, we used that
$$
\Big\lbrace x \in H_2(\widetilde{X}) \mid Q_{\widetilde{X}}(x,y)=0  \text{ for all } y \in H_2(\widetilde{X}) \Big\rbrace
=
\im(i_* \colon H_2(Y^{\varphi}) \to H_2(\widetilde{X})).
$$
The left hand side of this equation is referred to as the \emph{radical} of~$Q_{\widetilde{X}}$ and is abbreviated as~$\operatorname{rad}$.
Since~$\pi_1(\partial X) \to \pi_1(X)$ is surjective (so that~$H_3(\widetilde{X}) \cong H^1(\widetilde{X},\partial \widetilde{X})=0$), it follows that $i_*$ is injective, whence $\rad \cong H_2(Y^{\varphi})$.

We note that $\im(i_*)$ is also the radical of the $\Z[\pi]$-equivariant intersection form 
$$\lambda_X \colon H_2(\widetilde{X}) \times H_2(\widetilde{X}) \to \Z[\pi].$$
A \emph{stabilisation} of a $\Z[\pi]$-hermitian form~$(H,\lambda)$ refers to the hermitian form~$(H,\lambda)  \oplus \left( \Z[\pi]^2, \bsm 0&1\\ 1&0 \esm \right).$
We say that~$(H,\lambda)$ and~$(H',\lambda')$ are \emph{stably isometric} if they become isometric after each pair is stabilised~$n$ times for some~$n\geq 0$.
A $\Z[\pi]$-hermitian form~$(H,\lambda)$ is \emph{realisable} if there is a filling~$X$ of~$(Y,\varphi)$ with~$\pi_1(X) \cong \pi$ and~$(H_2(\widetilde{X}),\lambda_X) \cong (H,\lambda).$
A form~$(H,\lambda)$ is \emph{stably realisable} if $(H,\lambda)$ becomes realisable after some number of stabilisations.

The following proposition is similar to~\cite[Lemma~4.1]{HambletonKreck}.

\begin{proposition}
\label{prop:StablyRealisable}
Fix a closed~$3$-manifold~$Y$ and an epimorphism~$\varphi \colon \pi_1(Y) \twoheadrightarrow \pi$ to a finite group~$\pi$.
A $\Z[\pi]$-hermitian form~$(H,\lambda)$ is realisable by a filling of~$(Y,\varphi)$ if and only if it is stably realisable by a filling of~$(Y,\varphi)$.
\end{proposition}
\begin{proof}
A realisable form is certainly stably realisable, so we focus on the converse.
Assume there is a filling~$X$ of~$(Y,\varphi)$ with~$\pi_1(X) \cong \pi$,  and an isometry
$$(H,\lambda)   \oplus \left( \Z[\pi]^2, \bsm 0&1\\ 1&0 \esm\right)^{\oplus n} \cong (H_2(\widetilde{X}),\lambda_X).$$
Since finite groups are good (see e.g.~\cite[Theorem 19.2]{DET}) and $H_2(\widetilde{X}) \cong \pi_2(X)$, the sphere embedding theorem~\cite{FreedmanQuinn} now implies that~$X$ splits off~$n$ copies of~$S^2 \times S^2$; see e.g.~\cite[Theorem~2.3]{PowellRayTeichner}; the condition involving self-intersections can be omitted because~$X$ is orientable.
Surgering these connected summands yields a filling of~$(Y,\varphi)$ that realises~$(H,\lambda)$.
\end{proof}

\section{The exterior non-degenerate form}
\label{sec:ExteriorNonDegenerateForm}

This section introduces an algebraic construction that captures how, given a~$\Z_2$-surface~$F \subset S^4$ with exterior~$X_F$,  the  non-degenerate~$\Z$-intersection form on the universal cover of~$\widetilde{X}_F$ can be recovered from the intersection form of the double branched cover.

\begin{definition}
\label{def:ExteriorForm}
The \emph{exterior non-degenerate form}~$b_{ext}^{nd}$ associated to a non-singular symmetric bilinear form~$(H,b)$ over~$\Z$ refers to the pair~$(H_{ext}^{nd},b_{ext}^{nd})$, where
\begin{itemize}
\item the~$\Z[\Z_2]$-module~$H_{ext}^{nd}$ is defined, as an abelian group, as
$$H_{ext}^{nd} :=\Big\{ x \in H \mid b(x,x)=0 \mod 2 \Big\} \subset H$$
and is endowed with the~$\Z[\Z_2]$-module structure induced by~$Tx=-x$ for every~$x \in H_{ext}^{nd}$.
\item The non-degenerate symmetric bilinear form~$b_{ext}^{nd}$ is the restriction of~$b$ to~$H_{ext}^{nd}$.
\end{itemize}
\end{definition}

If a non-singular symmetric bilinear form $b$ is odd,  then the map $H \to \Z_2,x \mapsto b(x,x) \mod 2$ is surjective and leads to the short exact sequence
$$
0 \to H_{ext}^{nd}  \xrightarrow{\subset} H \xrightarrow{} \Z_2 \to 0.
$$
In particular, $(H_{ext}^{nd}, b_{ext}^{nd})$ is an index $2$ subform of $(H,b)$.

\begin{example}
\label{ex:ExteriorForm}
We illustrate this definition with some examples.
\begin{itemize}
\item If~$(H,b)=(H_2(\Sigma_2(F)),Q_{\Sigma_2(F)})$ is the~$\Z$-intersection form of the branched cover of a~$\Z_2$-surface~$F \subset S^4$, then, by~\cite[Proposition 5.10]{ConwayOrsonPowell},  the associated exterior form is the non-degenerate form~$Q_{\widetilde{X}_F}^{nd}$ on~$H_2(\widetilde{X}_F)/\operatorname{rad}$.
Here,  recall that~$\widetilde{X}_F$ denotes the universal cover of the exterior of~$F$.
We record the precise statement for later use.
First, ~\cite[Proposition 5.10]{ConwayOrsonPowell} shows that the inclusion $\widetilde{X}_F \hookrightarrow \Sigma_2(F)$ induces 
a form-preserving injection
$$ 
(H_2(\widetilde{X}_F)/\operatorname{rad},Q_{\widetilde{X}_F}^{nd})
\hookrightarrow
(H_2(\Sigma_2(F)),Q_{\Sigma_2(F)}).
$$
Then~\cite[Proposition 5.10]{ConwayOrsonPowell} shows that the image of this injection is~$(H^{nd}_{ext},b_{ext}^{nd})$.
\item If~$b$ is a non-singular symmetric bilinear form and~$H$ is a hyperbolic form over~$\Z$,  then 
$$
(b \oplus H)_{ext}^{nd} \cong b_{ext}^{nd} \oplus H.
$$
This is because hyperbolics are
even and therefore all squares are zero mod~$2$.
\item When~$(H,b)=(\Z^h,(1)^{\oplus h})$, the calculation from~\cite[Proof of Proposition~5.11]{ConwayOrsonPowell} gives
\begin{equation}
\label{eq:COPMatrix}
(H_{ext}^{nd},b_{ext}^{nd})=
	    \begin{pmatrix}
	    	4 & 2 & \cdots & & & 2 \\
	    	2 & 2 & 1 & \cdots & & 1 \\
	    	\vdots & 1 & 2 & 1 & \cdots & \vdots \\
	    	& \vdots & 1 & \ddots & & \\
	    	& & \vdots & & \ddots& 1 \\
	    	2 & 1 & \cdots & & 1 & 2 
	    \end{pmatrix}.
\end{equation}
In particular,  this describes $Q_{\widetilde{X}_F}^{nd}$ when~$F=U \subset S^4$ is the unknotted surface of non-orientable genus $h$ and extremal Euler number $e=-2h$.
\end{itemize}
\end{example}

The next proposition  records how this construction interacts with stabilisations.

\begin{proposition}
\label{prop:NonDegUnderStable}
If two non-singular symmetric bilinear forms~$b$ and~$b'$ are stably isometric, then their exterior forms~$(H_{ext}^{nd},b_{ext}^{nd})$ and~$((H_{ext}')^{nd},(b_{ext}')^{nd})$ are stably isometric.
\end{proposition}
\begin{proof}
Assume that~$b$ and~$b'$ are stably isometric, say~$b \oplus H \cong b' \oplus H'$ for some hyperbolic forms~$H$ and~$H'$.
It follows from the second item of Example~\ref{ex:ExteriorForm} that
$$b_{ext}^{nd} \oplus H
\cong (b \oplus H)_{ext}^{nd} 
\cong (b' \oplus H')_{ext}^{nd}
\cong (b')_{ext}^{nd} \oplus H'.$$
Here,  note that since the exterior forms are defined using the module structure induced by~$Tx=-x$ for every~$x$, this~$\Z$-isometry is automatically a~$\Z[\Z_2]$-isomorphism on the underlying~$\Z[\Z_2]$-modules.
This concludes the proof of the proposition.
\end{proof}

\section{Pullbacks of hermitian forms over~$\Z[\Z_2]$.}
\label{sec:Pullback}

This section is concerned with pullbacks.
Section~\ref{sub:Pullbacks} recalls a pullback construction due to Hambleton-Riehm~\cite{HambletonRiehm} in the setting of hermitian forms over~$\Lambda:=\Z[\Z_2]$.
Section~\ref{sub:PullbackSurfaceExterior} applies this construction to equivariant intersection forms of $\Z_2$-surface exteriors.

\subsection{Pullbacks}
\label{sub:Pullbacks}
We begin with some notation.

\begin{notation}
Given a~$\Lambda$-module~$M$,  consider the abelian groups~$M_+:= \{x\in M\mid Tx=x\}$ and~$M_-:=\{x\in M\mid Tx=-x\}$.
When $M$ has no $2$-torsion,  it fits into the pullback square
$$
\xymatrix@R0.4cm@C0.4cm{
	M \ar[r]\ar[d]& M/M_+\ar[d]. \\
M/M_- \ar[r]& M/(M_+,M_-):=M_2.
}
$$
When~$M=\Lambda$, we write~$\Z_\pm:=\Lambda/\Lambda_\mp=\Lambda/(1\mp T)$ for the abelian group $\Z$ endowed with the~$\Z_2$-action given by~$Tx=\pm x$ for every~$x \in \Z$.
The previous pullback square then reduces to 
$$
\xymatrix@R0.4cm@C0.4cm{
\Lambda \ar[r]\ar[d]& \Z_-\ar[d]. \\
\Z_+ \ar[r]& \Z_2.
}
$$
We also write~$\Lambda^\Q:=\Q[\Z_2]$ so that for~$M^\Q:=M \otimes_\Z \Q$,  the quotient maps induce an isomorphism
$$M^\Q \xrightarrow{\cong} M^\Q/M_-^\Q \oplus M^\Q/M_+^\Q.$$
In particular,  taking $M=\Lambda$, we record the isomorphism
\begin{align*}
\Lambda^\Q &\xrightarrow{\cong} 
\frac{\Lambda^\Q}{(1-T)} \oplus \frac{\Lambda^\Q}{(1+T)}=:\Q_+\oplus \Q_-.
\end{align*}
The inverse of this map is given by~$([a],[b]) \mapsto ((1+T)a+(1-T)b)/2.$
\end{notation}

Next, we move on from pullbacks of modules to pullbacks of hermitian forms.  In what follows, we will assume that~$M$ is torsion-free as an abelian group so that a theorem of Reiner ensures that~$M$ decomposes as a direct sum of~$\Z_+,\Z_-$ and~$\Lambda$-summands~\cite{Reiner}.

\begin{construction}\label{con:plusminusforms}
To a hermitian form~$\lambda \colon M \times M \to \Lambda$ on a $\Lambda$-module $M$ that is $\Z$-torsion-free as an abelian group,  we associate symmetric bilinear forms
\begin{align*}
&\lambda_+ \colon M/M_- \times M/M_- \to \Lambda/\Lambda_- \cong \Z_+, \\
&\lambda_- \colon M/M_+ \times M/M_+ \to \Lambda/\Lambda_+ \cong \Z_-,\ \text{and} \\
&\lambda_2 \colon M/(M_-,M_+) \times M/(M_-,M_+) \to \Lambda/(\Lambda_-,\Lambda_+)  \cong \Z_2.
\end{align*}
Tensor the form~$(M,\lambda)$ by~$\Q$ and denote the outcome by~$(M^\Q,\lambda^\Q)$.
This form is~$\Lambda^\Q:=\Q[\Z_2]$-valued.
The aforementioned quotient-induced isomorphism then leads to a decomposition
$$(M^\Q,\lambda^\Q) 
\cong 
(M^\Q/M_-^\Q,\lambda^\Q_+)
 \oplus 
 (M^\Q/M_+^\Q,\lambda^\Q_-) .$$
The situation is summarised by the following commutative diagram:
$$
\xymatrix@R0.5cm@R0.8cm{
M^\Q \times M^\Q  \ar[r]^{\lambda^\Q}\ar[d]^\cong& \Lambda^\Q \ar[d]^\cong \\
(M^\Q/M_-^\Q \oplus M^\Q/M_+^\Q)\times (M^\Q/M_-^\Q \oplus M^\Q/M_+^\Q)  \ar[r]^-{(\lambda^\Q_+ \ \lambda^\Q_-)}& \Q_+ \oplus \Q_-.
}
$$
Write~$\proj_{\pm} \colon M \to M/M_\pm$ for the canonical projections and set~$x_{\pm}:=\proj_{\mp}(x)$.
Since~$M$ is~$\Z$-torsion-free, restricting the pair of forms~$(\lambda^\Q_+,\lambda^\Q_-)$ to~$M/M_- \oplus M/M_+$ then leads to a pair~$(\lambda_+,\lambda_-)$ of $\Z$-valued forms that are related to $\lambda$ as follows:
\begin{equation}
\label{eq:HR1}
\lambda(x,y)=(\lambda_+(x_+,y_+),\lambda_-(x_-,y_-)).
\end{equation}
These forms descend to~$M_2:=M/(M_-,M_+)$ and induce a symmetric bilinear form on $M_2$.
\begin{equation}
\label{eq:HR2}
\lambda_2(x,y):=
[\lambda_+([x_+],[y_+])]
=[\lambda_-([x_-],[y_-])]
\in \Lambda/(\Lambda_-,\Lambda_+) \cong \Z_2.
\end{equation}
\end{construction}

We briefly comment on the notation used in~\eqref{eq:HR1}.

\begin{remark}
In~\eqref{eq:HR1}, we followed the notation of Hambleton-Riehm but it it is worth spelling out the meaning of this equation for later use. Indeed, the left hand side belongs to $\Lambda$ whereas the right hand side belongs to
$$ \Lambda_{pull}:=\left\lbrace
\begin{pmatrix} a \\ b \end{pmatrix} \in \frac{\Lambda}{(1-T)} \oplus \frac{\Lambda}{(1+T)} \mid \proj_{+}(a)=\proj_{-}(b) \right\rbrace.$$
The projections $\Lambda \to \Lambda/(1\pm T)$ are seen to induce an isomorphism $\Lambda \xrightarrow{\cong} \Lambda_{pull}$ whose inverse is given by~$(a,b) \mapsto ((1+T)a+(1-T)b)/2$; here the numerator of this expression is even because it can be written as $(a+b)+T(a-b)$, where~$a \pm b$ is even thanks to the condition $\proj_{+}(a)=\proj_{-}(b) $.
In particular, we record for later use that~\eqref{eq:HR1} can be rewritten as
\begin{equation}
\label{eq:HR1Rewrite}
\lambda(x,y)=\frac{(1+T)\lambda_+(x_+,y_+)+(1-T)\lambda_-(x_-,y_-)}{2}.
\end{equation}
\end{remark}

The following result is due to Hambleton and Riehm~\cite[Theorems 3 and 6]{HambletonRiehm}.
\begin{theorem}
\label{thm:HambletonRiehmTheorems36}
Fix a~$\Z$-torsion-free module~$\Lambda $-module~$M$ with associated pullback square
$$
\xymatrix@R0.4cm@C0.4cm{
	M \ar[r]^{\proj_{+}}\ar[d]_{\proj_{-}}& M/M_+\ar[d]. \\
M/M_- \ar[r]& M/(M_+,M_-):=M_2.
}
$$
The following assertions hold:
\begin{itemize}
\item Given a hermitian form~$(M,\lambda)$, there are unique symmetric bilinear forms~$(M/M_-,\lambda_+)$ and $(M/M_+,\lambda_-)$ satisfying~\eqref{eq:HR1} and a unique symmetric bilinear form~$(M_2,\lambda_2)$ satisfying~\eqref{eq:HR2}.
\item Given symmetric bilinear forms~$(M/M_-,\lambda_+),(M/M_+,\lambda_-),(M_2,\lambda_2)$ satisfying~\eqref{eq:HR2},  there exists a unique hermitian form~$\lambda$ on~$M$ satisfying~\eqref{eq:HR1}.
\end{itemize}
\end{theorem}

Since the form $(M_2,\lambda_2)$ is determined by $(M/M_+,\lambda_-),(M/M_-,\lambda_+)$ and~\eqref{eq:HR2}, we will frequently omit it from the notation.

\begin{example}
\label{ex:Hyperbolic}
Let $H(\Gamma)$ denote the standard hyperbolic form over $\Gamma$ where $\Gamma=\Z[\Z_2], \Z_-,\Z_+$ or~$\Z_2$.
We prove that the pullback of the hyperbolic forms~$H(\Z_+)$ and~$H(\Z_-)$ is the hyperbolic form~$H(\Z[\Z_2])$.
By Theorem~\ref{thm:HambletonRiehmTheorems36},  it suffices to prove that for~$\lambda:=H(\Z[\Z_2])$,  the~$\pm$ and~$2$-forms are respectively~$\lambda_{\pm}=H(\Z_{\pm})$ and~$\lambda_2=H(\Z_2)$.
This follows readily from the definition and from the fact that the quotient-induced isomorphism~$\Lambda^\Q \cong
\frac{\Lambda^\Q}{(1-T)} \oplus \frac{\Lambda^\Q}{(1+T)}$ maps~$1$ to~$([1],[1])$ and~$0$ to~$0$.
\end{example}

\begin{example}
	\label{ex:directsum}
	We prove that the pullback of a direct sum is the direct sum of the pullbacks. 
	 Namely, we prove that if~$(M,\lambda),(M/M_{\mp},\lambda_{\pm}),(M_2,\lambda_2)$ and~$(N,\kappa),(N/N_{\mp},\kappa_{\pm}),(N_2,\kappa_2)$ form two pullback squares, then the following is a pullback:
	\[
	\begin{tikzcd}
		(M,\lambda)\oplus (N,\kappa) \ar[r] \ar[d] & (M/M_+,\lambda_-)\oplus (N/N_+,\kappa_-) \ar[d] \\
		(M/M_-,\lambda_+)\oplus (N/N_-,\kappa_+) \ar[r] & (M_2,\lambda_2)\oplus (N_2,\kappa_2).
	\end{tikzcd}
	\]
Here,  we have used the canonical isomorphism $(M\oplus N)/(M\oplus N)_{\pm}\cong (M/M_{\pm})\oplus (N/N_{\pm})$.  By Theorem~\ref{thm:HambletonRiehmTheorems36}, it suffices to prove that $(\lambda\oplus\kappa)_{\pm}=\lambda_{\pm}\oplus\kappa_{\pm}$.  This follows from the fact that in this case the commutative diagram in Construction~\ref{con:plusminusforms} splits as a direct sum of diagrams.
\end{example}

The output of Example~\ref{ex:Hyperbolic} and Example~\ref{ex:directsum} is that the stabilisation of a pullback is the pullback of the stabilisations (apply Example~\ref{ex:directsum} with~$N=H(\Z[\Z_2])$ and use Example~\ref{ex:Hyperbolic} to compute~$N_{\pm}$ and~$N_2$).  This is one of the main facts that we will use in the proof of Theorem~\ref{thm:precise}.

We prove a condition on how to achieve an isometry~$(M,\lambda)\to (M',\lambda')$ of two pullback forms.

\begin{proposition}
\label{prop:pullback_isometry}
	Let~$(M,\lambda)$ and~$(M',\lambda')$ be two pullback forms and fix isometries
	\begin{align*}
		&\alpha_- \colon (M/M_+,\lambda_-) \xrightarrow{\cong} (M'/M'_+,\lambda'_-)
		&\alpha_+ \colon (M/M_-,\lambda_+) \xrightarrow{\cong} (M/M'_-,\lambda'_+).
	\end{align*}
Assume that $\alpha_-$ and $\alpha_+$ induce maps on $M/(M_-,M_+)$.
	If~$\alpha_+$ and~$\alpha_-$ agree on~$M/(M_-,M_+)$, then there exists an isometry~$\alpha\colon (M,\lambda)\to (M',\lambda')$.
\end{proposition}
\begin{proof}
Use the given isometries and apply the universal property of module pullbacks to obtain an isomorphism $\alpha$ fitting into the diagram
	\[
	\begin{tikzcd}
		(M,\lambda) \arrow[ddr,"\alpha_+\circ \proj_-"', bend right] \arrow[rrd,"\alpha_-\circ \proj_+", bend left] \arrow[dr,dashed, "\alpha"] & & \\
		& (M',\lambda') \arrow[d,"\proj_-"] \arrow[r,"\proj_+"] & (M'/M'_+,\lambda'_-) \ar[d] \\
		& (M'/M'_-,\lambda'_+) \arrow[r] & (M'_2,\lambda'_2).
	\end{tikzcd}
	\]
Since $\alpha_{\pm}$ are isometries, one can verify that the pushforward form~$\alpha_*(\lambda):=\lambda(\alpha(-),\alpha(-))$  is a form on $M'$ satisfying properties~\eqref{eq:HR1} and~\eqref{eq:HR2}.  
Hence, by the second item of Theorem \ref{thm:HambletonRiehmTheorems36}, $\alpha_*(\lambda)=\lambda'$, and so $\alpha$ is the desired isometry.
\end{proof}

\subsection{Pullbacks and equivariant intersection forms}
\label{sub:PullbackSurfaceExterior}

The goal of this section is to apply the pullback construction to equivariant intersection forms of $\Z_2$-surface exteriors.

In what follows,  recall that we denote the exterior of a $\Z_2$-surface $F \subset S^4$ by $X_F$.
Note that since the universal cover~$\widetilde{X}_F$ of~$X_F$ is simply-connected, the $\Z[\Z_2]$-module~$H_2(\widetilde{X}_F)$ is torsion-free as an abelian group, as required by the pullback construction.
Recall that the intersection form~$Q_{\widetilde{X}_F}$ on~$H_2(\widetilde{X}_F)$ descends to a non-degenerate symmetric bilinear form~$Q_{\widetilde{X}_F}^{nd}$ on~$H_2(\widetilde{X}_F)/\rad$.

\begin{proposition}
\label{prop:PullbackSurfaceExterior}
Given a $\Z_2$-surface $F \subset S^4$ of non-orientable genus $h$,  the pullback square associated to the hermitian form $(H,\lambda):=(H_2(\widetilde{X}_F),\lambda_{X_F})$ is 

$$
\xymatrix@R0.4cm@C0.4cm{
(H,\lambda)\ar[r]\ar[d]& (H/H_+,2Q_{\widetilde{X}_F}^{nd})\ar[d] \\
(H/H_-,0)\ar[r]& (H/(H_+,H_-),0).
}
$$
Here $H_+=\rad(\lambda)$ and, in addition,  there is an isomorphism $H \cong \Z_- \oplus \Z[\Z_2]^{h-1}$ which induces isomorphisms
$$
H/H_+ \cong \Z_-^h,  \quad  \quad H/H_- \cong \Z_+^{h-1}, \quad \text{ and } \quad H/(H_-,H_+) \cong \Z_2^{h-1}.
$$
	\end{proposition}
	\begin{proof}
Consider the projections~$\proj \colon H \to H/\rad$ and~$\proj_{\pm} \colon H \to H/H_\pm$.

We assert that~$\rad=H_+$ so that~$\proj_{+}=\proj$.
First, the inclusion~$\rad \subset H_+$ follows because~$\rad(\lambda_{X_F})=H_2(\partial \widetilde{X}_F) \cong \Z_+^{h-1}$; see e.g.~\cite[proof of Proposition 5.12]{ConwayOrsonPowell}.
We therefore obtain the following commutative diagram:
$$
\xymatrix@C1cm{
0\ar[r]&\rad\ar[r]\ar[d]^{\subset}&H\ar[r]^-{\proj}\ar[d]^=&H/\rad \ar[r]\ar@{->>}[d]&0\\
0\ar[r]&H_+\ar[r]&H\ar[r]^-{\proj_{+}}&H/H_+ \ar[r]&0.
}
$$
Since~\cite[Proposition 4.12]{ConwayOrsonPowell} ensures that~$H\cong \Z_- \oplus \Z[\Z_2]^{h-1}$,  we get~$H/H_+ \cong \Z_-^h$.
Using that~$H/\rad \cong \Z_-^h$ (thanks to~\cite[Proposition 5.12]{ConwayOrsonPowell}),  we deduce that the right hand vertical map is a quotient map from~$\Z^h_-$ that doesn't change the isomorphism type, and hence the quotient is trivial i.e.~$H_+\sm \rad=\emptyset$.
Hence we have that~$\rad=H_+$, concluding the proof of the assertion.

\begin{claim}
\label{claim:lambdaXF}
For $x,y \in H$, the following equality holds:
$$
\lambda_{X_F}(x,y)=(1-T)Q_{\widetilde{X}_F}^{nd}(x_-,y_-). 
$$
Here, recall the notation~$x_-:=\proj_+(x)$ and $y_-:=\proj_+(y).$
\end{claim}
\begin{proof}
[Proof of Claim~\ref{claim:lambdaXF}]
Write~$\lambda_{X_F}^{nd}$ for the $\Z[\Z_2]$-form induced by $\lambda_{X_F}$ on $H/\rad$.
Using the expression for $\lambda_{X_F}^{nd}$ derived in~\cite[Proposition 5.12]{ConwayOrsonPowell}, we obtain
\begin{equation}
\label{eq:StarterCalculation}
\lambda_{X_F}(x,y)
=\lambda_{X_F}^{nd}(\proj(x),\proj(y))
=(1-T)Q_{\widetilde{X}_F}^{nd}(\proj(x),\proj(y)).
\end{equation}
Combining the assertion with~\eqref{eq:StarterCalculation} implies that 
\begin{align*}
\lambda_{X_F}(x,y)
=(1-T)Q_{\widetilde{X}_F}^{nd}(\proj_{+}(x),\proj_{+}(y))
=(1-T)Q_{\widetilde{X}_F}^{nd}(x_-,y_-). 
\end{align*}
This concludes the proof of Claim~\ref{claim:lambdaXF}.
\end{proof}

We now verify that the proposed $+$ and $-$ forms satisfy~\eqref{eq:HR1} and~\eqref{eq:HR2}.
The result will then follow from Theorem~\ref{thm:HambletonRiehmTheorems36}.
For the first condition (rewritten as in~\eqref{eq:HR1Rewrite}), this follows because we get
\begin{align*}
\frac{(1-T)(\lambda_{X_F})_-(x_-,y_-)+(1+T)(\lambda_{X_F})_+(x_+,y_+)}{2} 
&=\frac{(1-T)2Q_{\widetilde{X}_F}^{nd}(x_-,y_-)+0}{2}\\
&=(1-T)Q_{\widetilde{X}_F}^{nd}(x_-,y_-).
\end{align*}
We now show~\eqref{eq:HR2} i.e.\ that $(H/\rad,2Q_{\widetilde{X}_F}^{nd})$ reduces to $(H/(H_+,H_-),0)$ modulo $2$.
This is clear because of the factor $2$ in the form.
The isomorphisms~$H/H_- \cong \Z_+^{h-1}$ and $H/(H_-,H_+) \cong \Z_2^{h-1}$ follow from the fact that $H \cong \Z_- \oplus \Z[\Z_2]^{h-1}$; see~\cite[Proposition 4.12]{ConwayOrsonPowell}.
\end{proof}

Even with the knowledge of $2Q_{\widetilde{X}_F}^{nd}$ from Example~\ref{ex:ExteriorForm}, Proposition~\ref{prop:PullbackSurfaceExterior} is challenging to use in practice unless one understands the maps in the pullback square, i.e.\ unless one is able to choose an explicit isomorphism~$H_2(\widetilde{X}_F) \cong \Z_- \oplus \Z[\Z_2]^{h-1}.$
Once one chooses such an isomorphism, the maps in this pulback square can be thought of as 
$$
\xymatrix@R0.4cm@C0.4cm{
(\Z_- \oplus \Z[\Z_2]^{h-1},\lambda)\ar[r]\ar[d]& (\Z_- \oplus \Z_-^{h-1},2Q_{\widetilde{X}_F}^{nd})\ar[d]^{\red_2 \circ \proj_2} \\
(\Z_+^{h-1},0)\ar[r]^{\red_2}& (\Z_2^{h-1},0).
}
$$
Here,  $\red_2$ denotes mod $2$ reduction and $\proj_2$ denotes projection onto the second factor.

The next proposition describes a type of isomorphism~$H_2(\widetilde{X}_F) \cong \Z_- \oplus \Z[\Z_2]^{h-1}$ for the unknot.
\begin{proposition}
\label{prop:IdentificationUnknot}
For $U:=U_h \subset S^4$ the unknotted surface with non-orientable genus~$h$ and (extremal) normal Euler number~$-2h$,  there exists an isomorphism
$$H_2(\widetilde{X}_U) \cong \Z_- \oplus \Z[\Z_2]^{h-1}$$
and a basis of the underlying free abelian group with respect to which $Q_{\widetilde{X}_U}$ is represented by
\begin{equation}
\label{eq:IntersectionFormHandle}
\begin{pmatrix} 4 & 2 & -2 & 2 & -2 & \dots & 2 & -2 \\ 2 & 2 & -2 & 1 & -1 & \dots & 1 & -1 \\ -2 & -2 & 2 & -1 & 1 & \dots & -1 & 1 \\ 2 & 1 & -1 & 2 & -2 & \dots & 1 & -1 \\ -2 & -1 & 1 & -2 & 2 & \dots & -1 & 1 \\ \vdots & \vdots & \vdots & \vdots & \vdots & \ddots & \vdots & \vdots \\ 2 & 1 & -1 & 1 & -1 & \dots & 2 & -2 \\ -2 & -1 & 1 & -1 & 1 & \dots & -2 & 2 \end{pmatrix}
\end{equation}
With respect to this basis, the induced form $Q_{\widetilde{X}_U}^{nd}$ on~$H_2(\widetilde{X}_U)/\rad$ is represented by the size $h$ matrix from~\eqref{eq:COPMatrix}.
\end{proposition}
\begin{proof}
We first obtain a handle diagram for $X_U$ depicted in the left hand side of Figure~\ref{fig:ExteriorAndCover} via the standard decomposition of $\R P^2$; see e.g.~\cite[Section 1]{KatanagaSaekiTeragaitoYamada} or~\cite[Figure 6.2]{GompfStipsicz} for a detailed discussion.
Perform isotopies to obtain the diagram on the right hand side of Figure~\ref{fig:ExteriorAndCover}, whence the diagram for the universal cover is readily obtained as shown in the bottom of Figure~\ref{fig:ExteriorAndCover}.

Label the cores of the 2-handles by $a_i$, $c_i$ for $1\leq i\leq h$ as in Figure~\ref{fig:ExteriorAndCover}.  
In this diagram the deck transformation is evident but, beyond the relations~$Ta_i=c_i$ and~$Tc_i=a_i$ for~$i=1,\ldots,h$, the module structure is not.

\begin{figure}[!htb]
\includegraphics[scale=0.5]{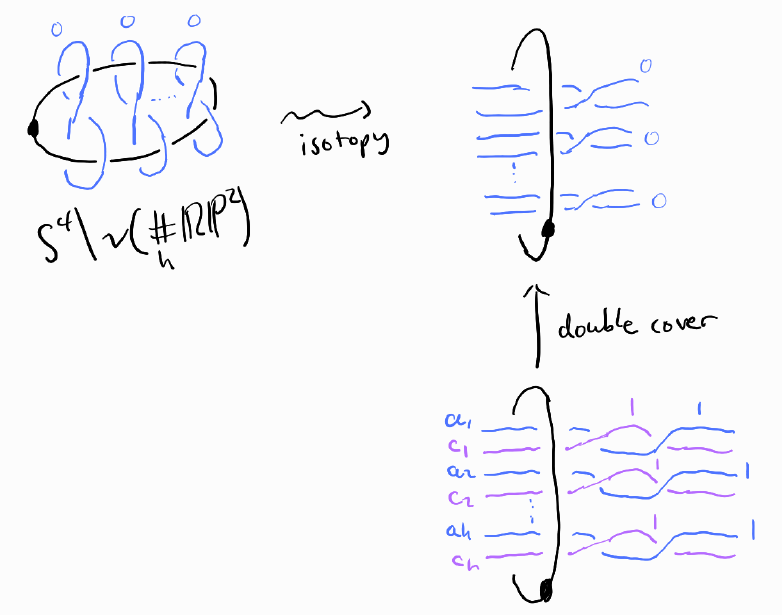}
\caption{
A handle diagram for the exterior $X_U$ as well as for its universal cover $\widetilde{X}_U$.
}
\label{fig:ExteriorAndCover}
\end{figure}

Form a basis for $H_2(\widetilde{X}_U)$ by setting~$e_1=c_1-a_1$ and~$e_i=a_i-c_1$ and~$f_i=c_i-a_1$ for~$i=2,\ldots,h$ (here when we take the difference of the cores we obtain spheres by adding the obvious bands between the attaching circles).
Observe that~$Te_1=-e_1$ and~$Te_i=f_i,Tf_i=e_i$ for~$i \geq 2$.
Perform handle sides to obtain the diagram on the right hand side of Figure~\ref{fig:FigureSlides}.
Removing the~$2$-handle labelled~$c_1$ and the~$1$-handle leads to a new~$4$-manifold but with the same intersection form, namely the one displayed in~\eqref{eq:IntersectionFormHandle}.
This is also illustrated in Figure~\ref{fig:FigureSlides}.
As explained above,  the~$e_i$ then give rise to the decomposition~$H_2(\widetilde{X}_U) \cong \Z_- \oplus \Z[\Z_2]^{h-1}$.

\begin{figure}[!htb]
\includegraphics[scale=0.7]{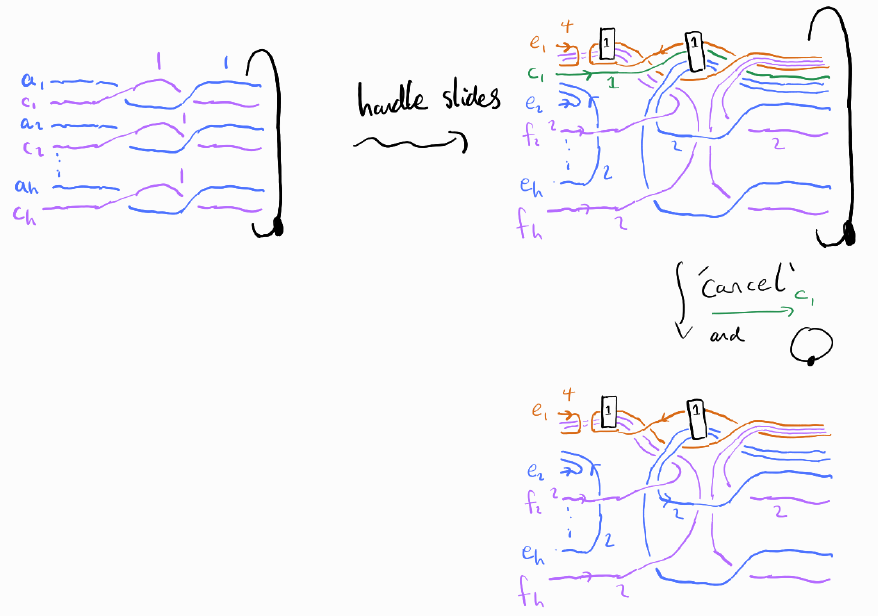}
\caption{
Left to right: Performing handle slides on the initial diagram for~$\widetilde{X}_U$.
Top to bottom: Forgetting the~$2$-handle~$c_1$ and the~$1$-handle leads allows for a prompt calculation of the intersection form of~$\widetilde{X}_U$ with respect to the basis~$e_1,e_2,f_2,\ldots,e_h,f_h$. 
The boxes denote a positive full twist.
}
\label{fig:FigureSlides}
\end{figure}

It only remains to calculate the nondegenerate form.
The radical of the matrix in~\eqref{eq:IntersectionFormHandle} is freely generated by~$e_i+f_i$ for~$i>2$, which agrees with the fact that the radical is endowed with the~$+$-structure.
In the quotient we obtain the relation~$[e_i]=-[f_i]$ for~$i=2,\ldots,h$.
In the basis~$[e_1],[e_2],[e_3],\ldots,[e_h]$,  we then obtain the required intersection form.
\end{proof}

\begin{convention}
\label{conv:Identification}
Given the unknotted~$\Z_2$-surface~$U\subset S^4$ of non-orientable genus $h$ and extremal Euler number $-2h$,  from now on, we fix a~$\Z[\Z_2]$-isomorphism~$H_2(\widetilde{X}_{U}) \cong \Z_- \oplus \Z[\Z_2]^{h-1}$ as in Proposition~\ref{prop:IdentificationUnknot} to identify the therein pullback square with the pullback square
$$
\xymatrix@R0.4cm@C1.4cm{
(\Z_- \oplus \Z[\Z_2]^{h-1},\lambda_{X_U})\ar[r]^{\id \oplus \red_{1+T}}\ar[d]^{\red_{1-T} \circ \proj_2}& (\Z_- \oplus \Z_-^{h-1},2Q_{\widetilde{X}_{U}}^{nd})\ar[d]^{\red_2 \circ \proj_2} \\
(\Z_+^{h-1},0)\ar[r]^{\red_2}& (\Z_2^{h-1},0).
}
$$
Here, we write $2Q_{\widetilde{X}_F}^{nd}$ as a shorthand for the matrix in~\eqref{eq:COPMatrix}.
We also recall that $\red_2$ denotes reduction mod $2$, that $\proj_2$ is the projection onto the second summand, and that $\red_{1 \pm T}$ denotes reduction modulo the ideal generated by~$1 \pm T$.
\end{convention}

\section{Constructing non-smoothable surfaces}
\label{sec:Proof}

This section is devoted to the proof of Theorem~\ref{thm:precise}.
We briefly outline the strategy of the proof as well as the organisation of the section.  
Start by using Examples~\ref{ex:Hyperbolic} and~\ref{ex:directsum} to consider the twice stabilised pullback square associated to the unknotted surface~$U_h \subset S^4$ with non-orientable genus~$h$ and (extremal) normal Euler number~$-2h$ (recall  Proposition~\ref{prop:PullbackSurfaceExterior}).
Setting $u:= Q_{\Sigma_2(U_h)}$ for the diagonal definite non-singular form of rank~$h$, this pullback takes the following form:
\begin{equation}
	\label{eq:Pullback1}
	\xymatrix@R0.4cm@C0.4cm{
{\overbrace{(H_2(\widetilde{X}_{U_h}),\lambda_{X_{U_h}})}^{:=(M_{U_h},\lambda_{U_h})}}\oplus H(\Z[\Z_2])^{\oplus 2} \ar[r]\ar[d]& 
(\Z_-^h,2u^{nd}_{ext})\oplus H(\Z_-)^{\oplus 2}\ar[d]^{\red_2 \circ \proj_2}. \\
		(\Z_+^{h-1},0)\oplus H(\Z_+)^{\oplus 2} \ar[r]^{{\red_2}}& (\Z_2^{h-1},0)\oplus H(\Z_2)^{\oplus 2}.
	}
\end{equation}
The first summands of the bottom horizontal map and the right vertical map are defined to be the ones from Convention~\ref{conv:Identification}.  On the hyperbolic summands, the horizontal (resp.\ vertical) maps are induced by quotienting by~$1+T$ (resp.\ ~$1-T$) as in Example~\ref{ex:Hyperbolic}.

Now fix an odd definite non-singular symmetric bilinear form~$b \colon \Z^h \times \Z^h \to \Z$.
This is the form that we wish to realise as the intersection form of the double branched cover of a $\Z_2$-surface.
Use Examples~\ref{ex:Hyperbolic} and~\ref{ex:directsum} to form the twice stabilised pullback square associated to~$b$ (which currently has no surface attached to~it):
\begin{equation}
\label{eq:Pullback2}
	\xymatrix@R0.4cm@C0.4cm{
		(M,\lambda_b)\oplus H(\Z[\Z_2])^{\oplus 2} \ar[r]\ar[d]& (\Z_-^h,2b^{nd}_{ext})\oplus H(\Z_-)^{\oplus 2} \ar[d]^{\red_2 \circ \proj_2} \\
		(\Z_+^{h-1},0)\oplus H(\Z_+)^{\oplus 2} \ar[r]^{\red_2}& (\Z_2^{h-1},0)\oplus H(\Z_2)^{\oplus 2}.
	}
\end{equation}
The bottom horizontal map and the right vertical map are defined to be the same as in~\eqref{eq:Pullback1}.
The main technical step in the proof of Theorem~\ref{thm:precise} consists of stably realising~$\lambda_b$.
This is achieved in Proposition~\ref{prop:alpha} but here we state the main steps.

In order to stably realise~$\lambda_b$,  Proposition~\ref{prop:alpha-} first builds an isometry 
$$\beta\colon (\Z_-^h,2u^{nd}_{ext})\oplus H(\Z_-)^{\oplus 2}\to(\Z_-^h,2b^{nd}_{ext})\oplus H(\Z_-)^{\oplus 2} .
$$
When $b \cong b' \oplus (1)$ with $b'$ an even definite non-singular symmetric bilinear form,  Proposition~\ref{prop:alpha-Fix} then shows that $\beta$ can be modified so as to yield a new isometry $\alpha_-$ that descends to an isometry 
$$
\alpha_2 \colon (\Z_2^{h-1},0)\oplus H(\Z_2)^{\oplus 2} \to (\Z_2^{h-1},0)\oplus H(\Z_2)^{\oplus 2}.
$$
Proposition~\ref{prop:alpha+} then lifts $\alpha_2$ to an isometry
$$
\alpha_+ \colon (\Z_+^{h-1},0)\oplus H(\Z_+)^{\oplus 2} \to (\Z_+^{h-1},0)\oplus H(\Z_+)^{\oplus 2}.
$$
The stable realisability of~$\lambda_b$ then follows from the work of Hambleton-Riehm (Proposition~\ref{prop:pullback_isometry}).
The construction of~$\alpha_-,\alpha_+$ and the proof of stable realisation are carried out in Sections~\ref{sub:alpha-}-\ref{sub:StablyRealising}.

Once~$\lambda_b$ is stably realised, the remaining steps are carried out in Section~\ref{sub:ProofEnd}.
In a nutshell,  Proposition~\ref{prop:StablyRealisable} ensures that $\lambda_b$ is in fact realised and it is then relatively routine to deduce that the form in fact arises as the equivariant intersection form of the exterior of a~$\Z_2$-surface $F \subset S^4$.
We then show that $Q_{\Sigma_2(F)} \cong b$.  This step is somewhat technical, since, due to the nature of our construction, we can only a priori deduce that $Q_{\Sigma_2(F)}$ and $b$ share an index two subform.

\subsection{Building the stable isometry $\alpha_-$ of the minus forms}
\label{sub:alpha-}

Recall that, as above,  we write~$U:=U_h \subset S^4$ for the unknotted surface of non-orientable genus $h$ with (extremal) normal Euler number~$-2h$ and set $u:=Q_{\Sigma_2(U_h)}$.

\begin{convention}
\label{convention:Fix}
The identification of Convention~\ref{conv:Identification} induces an isomorphism
$$M_U/(M_U)_+ \cong \Z_- \langle x \rangle \oplus \Z_-^{h-1}.$$
As a consequence, the map~$M_U/(M_U)_+ \to M_U/((M_U)_+,(M_U)_-)$ can be viewed as 
\begin{equation}
\label{eq:Proj-to2}
 \Z_- \oplus \Z_-^{h-1} \xrightarrow{\proj_2} \Z_-^{h-1} \xrightarrow{\red_2} \Z_2^{h-1}.
\end{equation}
With respect to the resulting basis, the first basis vector of $(\Z^h,u^{nd}_{ext})$ is $e_1:=2x$.
Note that $e_1$ has length $4$ and divisibility~$2$ in~$u^{nd}_{ext}$;  here the \emph{divisibility} $\operatorname{div}(y)$ of an element $y$  with respect to a symmetric bilinear form $(V,b)$ is the generator of the ideal $\langle b(y,z)\in \Z \mid z \in V \rangle$.
\end{convention}

\begin{lemma}\label{lem:induce}
Given an odd definite non-singular symmetric bilinear form~$b\colon \Z^h\times \Z^h\to \Z$ and an isometry 
	\[
	\alpha_-\colon (\Z_-^h,2u^{nd}_{ext})\oplus H(\Z_-)^{\oplus 2}\to(\Z_-^h,2b^{nd}_{ext})\oplus H(\Z_-)^{\oplus 2},
	\] 
the following assertions are equivalent:
\begin{itemize}
\item The isometry $\alpha_-$ 	induces an isometry 
	\[
	\alpha_2 \colon (\Z_2^{h-1},0) \oplus H(\Z_2)^{\oplus 2} 
	\to (\Z_2^{h-1},0) \oplus H(\Z_2)^{\oplus 2}.
	\]
\item The isometry $\alpha_-$ satisfies $\alpha_-(e_1)=ne_1+2\varepsilon$, where $n \in \Z$ is an integer, $e_1$ is the generator of the first $\Z_-$ summand,  and $\varepsilon\in \Z^{h+4}$.
\end{itemize}
\end{lemma}
\begin{proof}
Recalling the definition of the projection map from~\eqref{eq:Proj-to2},
a short diagram chase shows that~$\alpha_-$ induces an isomorphism $\alpha_2$ as in the statement if and only if~$\alpha_-$ preserves the kernel of the map~$\red_2\circ\proj_2\colon \Z_-\oplus \Z_-^{h+3}\to \Z_2^{h+3}$.  
	This kernel is generated by~$e_1$ along with all multiples of two.  
	The fact that~$\alpha_-$ preserves multiples of two follows from it being a~$\Z_-$-module isomorphism.  Hence the only condition is that~$\alpha_-(e_1)$ must lie in the kernel, i.e.~$\alpha_-(e_1)=n e_1+2\varepsilon$ for~$n$ some integer and~$\varepsilon\in \Z^{h+4}$.  All that remains to check is that this~$\alpha_2$ is an isometry, but this follows since the forms are zero on the first~$h-1$ summands and since~$\alpha_-$ was an isometry with respect to the hyperbolic form~$H(\Z_-)^{\oplus 2}$ on the last~$4$ summands.
\end{proof}

The following result is the first step in the construction of $\alpha_-$.

\begin{proposition}
\label{prop:alpha-}
If $b \colon \Z^h \times \Z^h \to \Z$ is an odd definite non-singular symmetric bilinear form,  then there exists a $\Z[\Z_2]$-isometry
	\[
\beta\colon (\Z_-^h,2u^{nd}_{ext})\oplus H(\Z_-)^{\oplus 2}\to(\Z_-^h,2b^{nd}_{ext})\oplus H(\Z_-)^{\oplus 2}. 
	\] 
\end{proposition}
\begin{proof}
	By the classification of non-singular indefinite symmetric bilinear forms (see e.g.~\cite{Serre}),  the~$\Z$-forms~$b$ and~$u=(1)^{\oplus h}$ become isometric after a single stabilisation (these stabilisations are indefinite, have the same rank, parity, and signature).
	  
	By Proposition~\ref{prop:NonDegUnderStable}, the forms~$b^{nd}_{ext}$ and~$u^{nd}_{ext}$ are stably isometric.  However, we need a stable isometry of the forms~$2b^{nd}_{ext}$ and~$2u^{nd}_{ext}$.  
	By the classification of non-degenerate forms due to Nikulin~\cite[Corollary~1.13.4]{Nikulin},~$2b^{nd}_{ext}$ and~$2u^{nd}_{ext}$ are stably isometric (after one stabilisation) if and only if they have the same rank, signature and quadratic boundary linking form (for a definition see e.g.~\cite[Section 7.4]{ConwayOrsonPowell}).  
	The first two follow readily from the fact that~$b^{nd}_{ext}$ and~$u^{nd}_{ext}$ are stably isometric.  The third also follows from this,  but the argument is  more involved.
	\begin{claim}
	\label{claim:BoundaryLinking}
		There exists an isometry of the quadratic linking forms of~$b^{nd}_{ext}$ and~$u^{nd}_{ext}$.  Furthermore, any such isometry~$\psi$ lifts to an isometry~$\widetilde{\psi}$ of the quadratic linking forms of $2b^{nd}_{ext}$ and $2u^{nd}_{ext}$.
	\end{claim}
	\begin{proof}[Proof of Claim~\ref{claim:BoundaryLinking}]
This claim can be proved using the definition of the quadratic boundary linking form mentioned e.g. in~\cite[Definition 7.26]{ConwayOrsonPowell} but we choose to work with matrices in order to avoid recalling background notions related to quadratic forms.

 Choose matrices~$A_b,A_u$ to represent the symmetric forms~$b^{nd}_{ext},u^{nd}_{ext}$ and~$Q_b,Q_u$ to represent the underlying quadratic forms, respectively, so that $A_u=Q_u+Q_u^T$ and $A_b=Q_b+Q_b^T$.
Use~$q_u \colon \Z^h/A_u\to \Q/\Z$ and~$q_b \colon \Z^h/A_b\to \Q/\Z$ to denote the quadratic boundary linking forms for~$b^{nd}_{ext}$ and~$u^{nd}_{ext}$, respectively.  
Taking inverses in~$\Q$,  as noted in~\cite[Remark 7.27]{ConwayOrsonPowell}, the  quadratic linking forms may be computed as 
\[
q_b(\pi(z))= z^T (A_b^{-1})^T Q_u A_b^{-1} z  \quad \text{ and } \quad q_u(\pi(z))= z^T (A_u^{-1})^T Q_u A_u^{-1} z.
\]
Here $\pi \colon \Z^h \to \Z^h/A_b$ denotes the canonical projection and similarly for $A_u$.

Similarly, the symmetric
forms~$2b^{nd}_{ext}$,~$2u^{nd}_{ext}$ and their underlying quadratic forms may then be represented by the matrices~$2A_b,2A_u,2Q_b,2Q_u$, respectively. 
This way, the corresponding quadratic linking forms~$q_{2u} \colon \Z^h/2A_u\to \Q/\Z$  and~$q_{2b} \colon \Z^h/2A_b\to \Q/\Z$ for~$2b^{nd}_{ext}$ and~$2u^{nd}_{ext}$ are given by 
\[
q_{2b}(\pi_2(z))= z^T ((2A_b)^{-1})^T 2Q_u (2A_b)^{-1} z  \quad \text{ and } \quad q_{2u}(\pi_2(z))= z^T ((2A_u)^{-1})^T 2Q_u (2A_u^{-1}) z.
\]
Here $\pi_2 \colon \Z^h \to \Z^h/2A_b$ denotes the canonical projection and similarly for $2A_u$.

Since $b^{nd}_{ext} \oplus H$ and~$u^{nd}_{ext}\oplus H$ are even and stably isometric,  so are their underlying quadratic forms.
This implies that their boundary quadratic linking forms are isometric.
Stabilising by a hyperbolic does not affect these linking forms,
so there exists a quadratic linking form isometry~$\psi\colon (\Z^h/A_u,q_u) \to (\Z^h/A_b,q_b)$. 
This establishes the first statement of the claim.  Now let $\psi$ be any such isometry.

We assert that the isomorphism $\psi \colon \Z^h/A_u \to \Z^h/A_b$ underlying this isometry lifts to an isomorphism~$\widetilde{\psi} \colon \Z^h/2A_u \to \Z^h/2A_u$ that makes the following diagram commute:
$$
\xymatrix{
\Z^h/2A_u \ar[d]_{\widetilde{\pi}}\ar@{-->}[r]^{\widetilde{\psi},\cong}& \Z^h/2A_b \ar[d]^{\widetilde{\pi}} \\
\Z^h/A_u \ar[r]^{\psi,\cong}& \Z^h/A_b.
}
$$
Here $\widetilde{\pi}$ denotes the relevant projection maps. 
 Recall that $A_u$ is congruent to the matrix displayed in~\eqref{eq:COPMatrix}.
Using row and column operations, one can show that this latter matrix (and therefore~$A_u$) has Smith normal form~$D=\operatorname{diag}(1,\ldots,1,4)$ if~$h$ is odd and to~$D=\operatorname{diag}(1,\ldots,1,2,2)$ if~$h$ is even.
Choose Smith normal forms
$$P_uA_uQ_u=D \quad \text{ and } \quad P_bA_bQ_b=D.$$
Set $G:=\coker(2D)$ for brevity from which it follows that~$2G \cong \coker(D)\in \{\Z_4,\Z_2 \oplus \Z_2\}.$
Putting~$A_b$ into Smith normal form and then performing the same row and column operations on~$2A_u$ and~$2A_b$ respectively leads to the following commutative diagram (in which the diagonal arrows arise from said operations,  the maps~$G \to 2G$ are the canonical ones, and the bottom horizontal map is the unique isomorphism that makes the diagram commute):
\begin{equation}
\label{eq:KeyDiagram}
\xymatrix{
&\Z^h/2A_u \ar[dd]_{\widetilde{\pi}} \ar[ld]_{P_u,\cong} & \Z^h/2A_b \ar[dd]^{\widetilde{\pi}} \ar[rd]^{P_b,\cong} \\
G\ar[dd]
\ar@{-->}@/^1pc/[rrr]^-{\cong}
&&&G\ar[dd] \\
&\Z^h/A_u \ar[r]^{\psi,\cong}\ar[ld]_{P_u,\cong}& \Z^h/A_b\ar[rd]^{P_b,\cong} \\
2G\ar@{-->}@/^1pc/[rrr]^{\cong}&&&2G. \\
}
\end{equation}
The assertion therefore reduces to proving that~$\Aut(G) \to \Aut(2G)$ is surjective.
Since~$2G$ is either~$\Z_4$ or~$\Z_2 \oplus \Z_2$, this is clear: in the first case because~$\Aut(\Z_4)=\{\pm 1\}$ and in the second because~$GL_2(\Z_4) \to GL_2(\Z_2)$ is surjective.
This concludes the proof of the assertion.

It remains to prove that the isomorphism~$\widetilde{\psi}$ from the assertion is an isometry of quadratic linking forms.
Consider the following commutative diagram:
$$
\xymatrix{
\Z^{h} \ar@{->>}[d]_{\pi_2}\ar@/_2pc/[dd]_{\pi}& \Z^{h} \ar@{->>}[d]^{\pi_2} \ar@/^2pc/[dd]^{\pi} \\
\Z^h/2A_u \ar[d]_{\widetilde{\pi}}\ar[r]^{\widetilde{\psi},\cong}& \Z^h/2A_b \ar[d]^{\widetilde{\pi}} \\
\Z^h/A_u \ar[r]^{\psi,\cong}& \Z^h/A_b.
}
$$
Choose $y \in \Z^h$ so that~$\pi_2(y)=\widetilde{\psi}(\pi_2(z))$.
It follows that~$\pi(y)=\psi(\pi(z))$ so that~$y$ can be used to calculate both $q_{2b}(\widetilde{\psi}(\pi_2(z)))$ and $q_b(\psi(\pi(z)))$.
We deduce that 
\begin{align*}
q_{2u}(\pi_2(z))
=& \frac{1}{2} z^T (A_u^{-1})^T Q_u A_u^{-1} z 
= \frac{1}{2} q_u(\pi(z))
=\frac{1}{2}q_b(\psi(\pi(z))) \\
&= \frac{1}{2} y^T (A_b^{-1})^T Q_b A_b^{-1} y
=q_{2b}(\widetilde{\psi}(\pi_2(z))).
\end{align*}
This concludes the proof of Claim~\ref{claim:BoundaryLinking}.
	\end{proof}
	
	We now appeal to the aforementioned classification of even non-degenerate symmetric bilinear forms due to Nikulin \cite[Corollary 1.13.4]{Nikulin} which says that $2b^{nd}_{ext}$ and $2u^{nd}_{ext}$ are stably isometric (after one stabilisation) since they have the same rank, signature, and quadratic boundary linking form. 
	The output is an isometry 
	\[
\beta \colon (\Z_-^h,2u^{nd}_{ext})\oplus H(\Z_-)^{\oplus 2}\xrightarrow{\cong}(\Z_-^h,2b^{nd}_{ext})\oplus H(\Z_-)^{\oplus 2}.
	\]
Here recall that $\beta$ is automatically a $\Z[\Z_2]$-isometry because the modules are endowed with the~$\Z_2$-action induced by $Tx=-x$. 
\end{proof}

The following result constructs the required isometry $\alpha_-$.

	\begin{proposition}
	\label{prop:alpha-Fix}
	If~$b=(1)\oplus b'\colon \Z^h\times \Z^h\to \Z$ for $b'$ an even definite non-singular symmetric bilinear form and there exists a~$\Z[\Z_2]$-isometry
	\[
	\beta \colon (\Z_-^h,2u^{nd}_{ext})\oplus H(\Z_-)^{\oplus 2}\xrightarrow{\cong}(\Z_-^h,2b^{nd}_{ext})\oplus H(\Z_-)^{\oplus 2},
	\]
	then there exists a~$\Z[\Z_2]$-isometry 
	\[
	\alpha_- \colon (\Z_-^h,2u^{nd}_{ext})\oplus H(\Z_-)^{\oplus 2}\xrightarrow{\cong}(\Z_-^h,2b^{nd}_{ext})\oplus H(\Z_-)^{\oplus 2}
	\]
	that descends to an isometry
	\[
	\alpha_2\colon (\Z_2^{h-1},0)\oplus H(\Z_2)^{\oplus 2}\xrightarrow{\cong}(\Z_2^{h-1},0)\oplus H(\Z_2)^{\oplus 2}.
	\]
\end{proposition}
\begin{proof}
	By Lemma~\ref{lem:induce} we have that~$\beta$ induces an isometry 
	\[
	(\Z_2^{h-1},0)\oplus H(\Z_2)^{\oplus 2}\xrightarrow{\cong}(\Z_2^{h-1},0)\oplus H(\Z_2)^{\oplus 2}
	\]
	if and only if~$\beta(e_1)=e_1 +2\varepsilon$ where $\varepsilon\in \Z_-^{h+4}$. 
	This may not be true for~$\beta$, but the aim is to precompose~$\beta$ with another isometry~$\varphi$ such that the composition~$\beta \circ \varphi$ does have this property by using Eichler's criterion~\cite[Section 10]{Eichler} (see also e.g.~\cite[Lemma~7.5]{Gritsenko}). 
	Essentially we will arrange for Eichler's criterion to hold for the pair~$e_1$ and~$\beta^{-1}(e_1)$ at the expense of having to add some~$2\varepsilon$ to~$e_1$.
	
	We work with the inverse isometry (and thus on~$2u^{nd}_{ext}$) to avoid having to work with the particulars of the form~$b$.
Write $A_u$ for the matrix from~\eqref{eq:COPMatrix} and choose a Smith normal form $P_uA_uQ_u=D$; here recall that~$D=\operatorname{diag}(1,\ldots,1,4).$
	
	Eichler's criterion requires that the primitive vectors~$e_1$ and~$\beta^{-1}(e_1)$ both have the same length,  divisibility with respect to~$2A_u \oplus H(\Z_-)^{\oplus 2}$,  and that they determine the same element in~$\Z^{h+4}/(2A_u\oplus H(\Z_-)^{\oplus 2})$.
	Here, \emph{the element determined by~$x$} in~$\Z^{h+4}/(2A_u\oplus H(\Z_-)^{\oplus 2})$ refers~to
	$$(2A_u\oplus H(\Z_-)^{\oplus 2})(x)/\operatorname{div}(x) \in \Z^{h+4}/(2A_u\oplus H(\Z_-)^{\oplus 2}) \cong \Z^h/2A_u.$$
In what follows, we write~$\wt{P}_u$ for the Smith normal form of~$2A_u\oplus H(\Z)^{\oplus 2}$ obtained by taking the direct sum of~$P_u$ with an arbitrary Smith normal form for~$H(\Z)^{\oplus 2}.$
In order to perform explicit calculations, we make an explicit choice of $P_u$, namely (for $h=9)$:
	\[
	P_u := \begin{pmatrix}
		0 & 1 & 0 & 0 & 0 & 0 & 0 & 0 & 0  \\
		0 & 0 & 1 & 0 & 0 & 0 & 0 & 0 & 0  \\
		0 & 0 & 0 & 1 & 0 & 0 & 0 & 0 & 0  \\
		0 & 0 & 0 & 0 & 1 & 0 & 0 & 0 & 0  \\
		0 & 0 & 0 & 0 & 0 & 1 & 0 & 0 & 0  \\
		0 & 0 & 0 & 0 & 0 & 0 & 1 & 0 & 0  \\
		0 & 0 & 0 & 0 & 0 & 0 & 0 & 1 & 0  \\
		0 & 0 & 0 & 0 & 0 & 0 & 0 & 0 & 1  \\
		1 & 2 & 2 & 2 & 2 & 2 & 2 & 2 & 2 
	\end{pmatrix}.
	\]
We will use the isomorphism~$\widetilde{P}_u=P_u \colon \Z^h/2A_u \xrightarrow{\cong} G$ to verify Eichler's third condition in~$G$.
Here, we are identifying~$G \oplus 0^{\oplus 2}$ with~$G$.
From now on, for $x \in \Z^{h+4}$, we write
$$[x]:=\widetilde{P}_u (2A_u \oplus H(\Z_-)^{\oplus 2})(x)/\divv(x) \in G.
$$
Note that~$e_1$ has length~$8$, divisibility~$4$ and that, the corresponding element in~$G$ is 
	$$[e_1]=\widetilde{P}_u (2A_u \oplus H(\Z_-)^{\oplus 2})(e_1)/4=(1,1,1,\dots,1,\pm 2)\in G.$$
	We now use the assumption that $b=(1)\oplus b'$ where $b'$ is even.  We deduce that $2b^{nd}_{ext}\cong (8)\oplus 2b'$.  Fix a matrix representative $A_{b'}$ for $b'$, hence $2A_b=(8)\oplus 2A_{b'}$ is a representative for $2b^{nd}_{ext}$.  We then fix $P_b$ a Smith normal form matrix for $2A_b$ given as
	\[
	P_b = \begin{pmatrix}
		0 & A_{b'}^{-1} \\
		1 & 0
	\end{pmatrix}.
	\]
Writing~$P_b(2A_b)=DQ_b$, a rapid verification shows that~$\det(P_b(2A_b))=\pm \det(D)$ so that~$Q_b$ is indeed invertible over~$\Z$.
In what follows, we write~$\wt{P}_b$ for the Smith normal form for~$2A_b\oplus H(\Z_-)^{\oplus 2}$ obtained by taking the direct sum of~$P_b$ with an arbitrary Smith normal form for~$H(\Z_-)^{\oplus 2}.$

	By inspection, $e_1$ has length and divisibility both equal to $8$ in $2A_b$.  Using~$x_1,y_1$ to denote the generators of the first hyperbolic summand,~$e_1+4x_1$ is seen to still be primitive and have divisibility~$4$.
	Also, since~$x_1$ has length $0$, the length of~$e_1+4x_1$ is still $8$. 
    It follows that 
    \[\widetilde{P}_b (2A_b \oplus H(\Z_-)^{\oplus 2})(e_1+4x_1)/4=(0,0,\dots,0, 2) \in G.\]
	Now we swap to considering~$\beta^{-1}(e_1+4x_1)$ which we denote~$e_1'$. 
	Since~$\beta$ is an isometry,~$e_1'$ is still primitive with divisibility~$4$ and length $8$.  Furthermore, since~$\beta^{-1}$ lifts~$\pm\id$ on~$G$,  and since $\wt{P}_u,\wt{P}_b$ respectively agree with $P_u,P_b$ on $G$, it follows that
    $$[e_1']
    :=\widetilde{P}_u (2A_u \oplus H(\Z_-)^{\oplus 2})(e_1')/4
=\pm \widetilde{P}_b (2A_b \oplus H(\Z_-)^{\oplus 2})(e_1+4x_1)/4
    =(0,0,\dots,0,\pm 2).$$
We first arrange for Eichler's third condition to hold; after that we will arrange for the length condition.  At every step we avoid changing the divisibility.  Up to changing the sign of~$\beta$ (by changing the sign of the initial isomorphism~$\psi$ in the proof of Proposition~\ref{prop:alpha-}), this means that
    $$[e_1']-[e_1]=(1,1,\dots,1,0).$$
    We need to kill this difference by adding even elements: we do this to $e_1$.
To do so, consider for~$i=1,\ldots,h-1$,
	 the vector $v_i=(2,2,\dots, 2, 0, 2, \dots, 2,0,0,0,0) \in \Z^{h+4}$, where the first zero is in the~$(i+1)$-th slot. 
    A short calculation shows that these satisfy 
    $$[v_i]=(0,0,\dots,1,\dots,0,\ell) \in G$$ 
with the $1$ in the~$i$-th slot and where~$\ell=0$ or~$4$.  
Observe that $[e_1+v_i]$ is still order $4$ in $G$ and, since $e_1+v_i \in \Z^{h+4}$ is primitive, it follows that its divisibility is $4$.
Since $v_i$ and $e_1$ also have divisibility $4$, we note that this implies~$[e_1+v_i]=[e_1]+[v_i]$.
	 By adding each of these vectors in turn we obtain that $[e_1+\sum_i v_i]=[e_1']$, again up to changing the sign of $\beta$.

At this stage, we have established Eichler's second and third conditions for $e_1'$ and $e_1 + \sum_i v_i$, and so it remains to further modify the latter vector so as to arrange the first of Eichler's conditions which, recall, concerns length.

We assert that the lengths of $e_1$ and $e_1+\sum_i v_i$ differ by a multiple of $32$.
For brevity, we write~$x \cdot y:=x^T(2A_u)y,$ so that a routine calculation shows that for all $i$
	\begin{equation}\label{eq:cal1}
		2e_1\cdot v_i + v_i\cdot v_i\equiv 16 \pmod{32}
	\end{equation}
	and for all $i<j$
	\begin{equation}\label{eq:cal2}
		2v_i\cdot v_j\equiv 16\pmod{32}.
	\end{equation}
	We then calculate
	\begin{align*}
        \left(e_1+\sum_i v_i\right)\cdot \left(e_1+\sum_i v_i\right)- e_1\cdot e_1 &= 2e_1\cdot \sum_i v_i + \sum_i v_i\cdot \sum_i v_i \\
		&= \left(\sum_i 2e_1\cdot v_i + v_i\cdot v_i\right) + \left(\sum_{i<j} 2v_i\cdot v_j\right) \\
		&\equiv  16(h-1) + 16\binom{h-1}{2}  \pmod{32}\\
		&= 16\binom{h}{2} = 16\frac{h(h-1)}{2} \\
		&\equiv 0 \pmod{32}.
	\end{align*}
	The first congruence follows from (\ref{eq:cal1}) and (\ref{eq:cal2}). 
    The following equality follows from Pascal's triangle.
    The final congruence follows since $h-1$ is divisible by four (since $b'$ must have rank divisible by eight).
This concludes the proof of the assertion.
	
	Hence the length of~$e_1+\sum_i v_i$ differs from the length of~$e_1$ by a multiple of 32, say $32n$. Now observe that~$e_1+\sum_i v_i+(4nx_1-4y_1)$ is a vector with the desired length, 
    divisibility 
    and satisfies 
    \[ \left[e_1+\sum_i v_i+(4nx_1-4y_1)\right]=[e'_1].\] 
Apply Eichler's criterion to obtain the desired isometry $\varphi$ with~$\varphi(e_1+\sum_i v_i+(4nx_1-4y_1))=e_1'$.
Here, recall that $e_1':=\beta^{-1}(e_1+4x_1)$.
Set~$\alpha_-:=\beta \circ \varphi$.  A short calculation using the linearity of these isometries verifies that $\alpha_-(e_1)=e_1+2\varepsilon$ for some $\varepsilon$.
\end{proof}

\subsection{Building the stable isometry $\alpha_+$ of the plus forms}
\label{sub:alpha+}

Proposition~\ref{prop:alpha-Fix} ensures the existence of an isometry 
	\[
	\alpha_-\colon (\Z_-^h,2u^{nd}_{ext})\oplus H(\Z_-)^{\oplus 2}\xrightarrow{\cong} (\Z_-^h,2b^{nd}_{ext})\oplus H(\Z_-)^{\oplus 2}.
	\] 
	
The goal of this section is to show that if $\alpha_-$ descends to an isometry $\alpha_2$, then the latter lifts to an isometry~$\alpha_+$.

\begin{proposition}
\label{prop:alpha+}
Fix an odd definite non-singular symmetric bilinear form~$b \colon \Z^h \times \Z^h \to \Z$.
If an isometry
	\[
	\alpha_-\colon (\Z_-^h,2u^{nd}_{ext})\oplus H(\Z_-)^{\oplus 2}\xrightarrow{\cong}(\Z_-^h,2b^{nd}_{ext})\oplus H(\Z_-)^{\oplus 2}.
	\] 
descends to an isometry $\alpha_2$, then there exists an isometry 	
$$ \alpha_+ \colon (\Z_+^{h-1},0)\oplus H(\Z_+)^{\oplus 2} \xrightarrow{\cong} (\Z_+^{h-1},0)\oplus H(\Z_+)^{\oplus 2} $$ 
that induces the same isometry as~$\alpha_-$ on 
$$
 (\Z_2^{h-1},0)\oplus H(\Z_2)^{\oplus 2} \xrightarrow{\cong}  (\Z_2^{h-1},0)\oplus H(\Z_2)^{\oplus 2} .
$$
\end{proposition}
\begin{proof}
We view this as a lifting problem.
Consider the diagram
	\[
	\begin{tikzcd}
(\Z_+^{h-1},0)\oplus H(\Z_+)^{\oplus 2} \ar[r,dotted, "\alpha_+"]\ar[d,"\red_2\oplus \red_2"]&
(\Z_+^{h-1},0)\oplus H(\Z_+)^{\oplus 2}\ar[d,"\red_2\oplus \red_2"] \\
		(\Z_2^{h-1},0)\oplus H(\Z_2)^{\oplus 2} \ar[r,"{\alpha_2,\cong}"]& 
		(\Z_2^{h-1},0)\oplus H(\Z_2)^{\oplus 2}.
	\end{tikzcd}
	\]  
That the vertical maps are as we claim follows from Convention~\ref{conv:Identification} for the first summand and Example~\ref{ex:Hyperbolic} for the second summand.

We wish to lift $\alpha_2$ to an isometry $\alpha_+$. 
The underlying map $\alpha_2$ is determined by a matrix with entries in $\Z_2$ that we write as the block matrix
 \[\begin{pmatrix}
		A & B \\
		C & D
	\end{pmatrix}\]
	where the matrix $D$ is a $4\times 4$ block.  
A rapid calculation using that $\alpha_2$ is an isometry is seen to imply that $D$ represents an isometry~$H(\Z_2)^{\oplus 2}\to H(\Z_2)^{\oplus 2}$. 
Another rapid calculation also shows that since $\alpha_2$ is an isometry and~$D$ is invertible, $C$ must be the zero matrix:~$C=\mathbf{0}$.  
Since~$\alpha_2$ is an isomorphism, it follows that~$A$ is invertible.

	  \begin{claim}
	  \label{claim:Forbidden}
Writing $H:=H(\Z_+)^{\oplus 2}$ and $Q=\bsm 0&1 \\ 0&0\esm^{\oplus 2}$, the matrix $D$ satisfies~$D^TH^{\oplus 2}D=H^{\oplus 2} \mod 2$ and its columns satisfy~$c^TQc=0 \mod 2$.
	  \end{claim}
	  \begin{proof}[Proof of Claim~\ref{claim:Forbidden}]
We have already argued that $D$ preserves $H$ so we focus on the second assertion.
Write~$(\lambda_u)_-^{stab}:=2u_{ext}^{nd} \oplus H(\Z_-)^{\oplus 2}$ and~$(\lambda_b)_-^{stab}:=2b_{ext}^{nd} \oplus H(\Z_-)^{\oplus 2}$ for brevity.
From its definition we deduce that $b^{nd}_{ext}$ has diagonal entries even, and so $2b^{nd}_{ext}$ has diagonal entries divisible by four and off-diagonal entries even.
Use~$e_1,f_1,e_2,f_2$ to denote the canonical basis of the hyperbolic summand in $\Z_-^h\oplus H(\Z_-)^{\oplus 2}$.

Since the maps $H(\Z_{-})^{\oplus 2} \to H(\Z_2)^{\oplus 2}$ in~\eqref{eq:Pullback1} and~\eqref{eq:Pullback2} are given by reduction mod $2$,  the map~$\alpha_-$ can be represented by a $2\times 2$ block matrix whose lower right corner is given by a $4\times 4$ matrix of the form $(d_{ij})_{1\leq i,j\leq 4}$ that reduces mod $2$ to $D$.

With this notation, $\alpha_-(e_1) = (x \ d_{11} \ d_{21} \ d_{31} \ d_{41})^T$ for some $x \in \Z^h$. 
Note that $(\lambda_u)_-^{stab}(e_1, e_1) =~0$. 
Since $\alpha_-$ is an isometry, it follows that
\begin{align*}
0 &= (\lambda_b)_-^{stab}(\alpha_-(e_1), \alpha_-(e_1)) \\
&= x^T(2b^{nd}_{ext})x + 
\begin{pmatrix} 
d_{11} & d_{21} & d_{31} & d_{41} 
\end{pmatrix}
\begin{pmatrix} 
0 & 1 & 0 & 0 \\ 
1 & 0 & 0 & 0 \\ 
0 & 0 & 0 & 1 \\ 
0 & 0 & 1 & 0 
\end{pmatrix} 
\begin{pmatrix} 
d_{11} \\ d_{21} \\ d_{31} \\ d_{41} 
\end{pmatrix} \\
&= x^T(2b^{nd}_{ext})x + 2(d_{11}d_{21} + d_{31}d_{41}).
\end{align*}
Since $2b^{nd}_{ext}$ has diagonal entries that are multiples of $4$ and off-diagonal entries that are even, we deduce that $x^T(2b^{nd}_{ext})x \equiv 0 \pmod{4}$. 
It follows that $2(d_{11}d_{21} + d_{31}d_{41}) \equiv 0 \pmod{4}$, which forces the sum $d_{11}d_{21} + d_{31}d_{41}$ to be even. 
This implies that $(d_{11}, d_{21}, d_{31}, d_{41})Q(d_{11}, d_{21}, d_{31}, d_{41})^T=0 \mod 2$.
Repeating the argument with $f_1, e_2$ and $f_2$ sequentially in place of $e_1$ shows that all four column vectors of~$D$ satisfy $c^TQc=0$.
This concludes the proof of Claim~\ref{claim:Forbidden}.
	  \end{proof}

\begin{claim}
\label{claim:IntegralLift}
Every $4 \times 4$ matrix $X$ over $\Z_2$ that satisfies~$X^THX=H \mod 2$ and whose columns satisfy $c^TQc=0$ lifts to a $4 \times 4$ matrix $\widetilde{X}$ over $\Z$ with $\widetilde{X}^TH(\Z_+)^{\oplus 2}\widetilde{X}=H(\Z_+)^{\oplus 2}.$
\end{claim}	  
\begin{proof}
Consider the quadratic form~$q \colon \Z_2^4 \to \Z_2,x \mapsto x^TQx$ whose associated bilinear form is~$H$.
A verification shows that the subgroup of~$\operatorname{Sp}(4,\Z_2)$ of matrices whose columns satisfy~$c^TQc=0$ is precisely the subgroup~$O^+_4(2)$ of~$\operatorname{Sp}(4,\Z_2)$ that consists of matrices~$M$ with~$q(Mx)=q(x)$ for every~$x$.
This latter group is known to have $72$ elements~\cite[Proposition 2.5.3(i), Equation~(2.5.4) and Proposition 2.5.5]{KleidmanLiebeck}

We begin our construction of the lifts of the elements of~$O_4^+(2)$.
First, to every integral matrix~$A =\bsm a&b \\ c &d \esm \in SL_2(\Z)$,  we associate the integral matrices
$$
L_A=
\begin{pmatrix}
a&0&0&-b \\
0&d&c&0\\
0&b&a&0 \\
-c&0&0&d
\end{pmatrix}
\quad \text{and} \quad
R_A=
\begin{pmatrix}
d&0&-c&0 \\
0&a&0&b\\
-b&0&a&0 \\
0&c&0&d
\end{pmatrix}.
$$
We also consider the matrix 
$$T
=\begin{pmatrix}
1&0&0&0 \\
0&1&0&0\\
0&0&0&-1 \\
0&0&-1&0
\end{pmatrix}.
$$
The matrices~$L_A,R_A$ and~$T$ are seen to preserve (the integral version of) $H$.
Next, we consider the following integral lifts of the $6$ elements of $SL_2(\Z_2)$:
	  \[
	  A_1:=\begin{pmatrix}
	  	1 & 0 \\
	  	0 & 1
	  \end{pmatrix},
	  A_2:=\begin{pmatrix}
	  	0 & 1 \\
	  	1 & 0
	  \end{pmatrix},
	  A_3:=\begin{pmatrix}
	  	1 & 1 \\
	  	0 & 1
	  \end{pmatrix},
	  A_4:=\begin{pmatrix}
	  	1 & 0 \\
	  	1 & 1
	  \end{pmatrix},
	  A_5:=\begin{pmatrix}
	  	1 & 1 \\
	  	1 & 0
	  \end{pmatrix},
	  A_6:=\begin{pmatrix}
	  	0 & 1 \\
	  	1 & 1
	  \end{pmatrix}.
	  \] 
Given two matrices $A_i,A_j \in SL_2(\Z)$ as above and $\varepsilon \in \{0,1\}$, we associate the matrix 
$$M(A_i,A_j,\varepsilon)=L_{A_i}R_{A_j}T^\varepsilon.$$
There are at most $6 \cdot 6\cdot 2= 72$ such matrices,  all of which preserve $H$ and hence this determines at most 72 elements of~$\operatorname{Sp}(4,\Z_2)$.  One also verifies that their columns satisfy~$c^TQc=0 \mod 2$ and hence it follows that the mod $2$ reduction of the~$M(A_i,A_j,\varepsilon)$ define elements of~$O^+_4(2)$.

It only remains to show that the assignment~$(A,B,\varepsilon) \mapsto M(A,B,\varepsilon)$ is injective.
Here recall that the matrices~$A,B \in \{A_1,\ldots,A_6\}$ are integral.
Establishing injectivity amounts to proving that~$L_A=R_BT^\varepsilon$ mod~$2$ implies~$A=I,B=I$ and~$\varepsilon=0$.
Note that~$(L_A)^{-1}=L_{A^{-1}}$ mod~$2$ and~$(R_A)^{-1}=R_{A^{-1}}$ mod~$2$.
When~$\varepsilon=0$,   if~$L_AR_B=I$ mod~$2$, then~$L_A=R_{B^{-1}}$ which is seen to imply that~$A=I=B^{-1}$ mod~$2$.
For~$A,B \in  \{A_1,\ldots,A_6\}$, this implies that~$A=I=B$.
The case~$\varepsilon=1$ leads to a contradiction,  thus concluding the proof of injectivity.
\end{proof}

Using Claim~\ref{claim:IntegralLift} we know that the matrix $D$ lifts to a isometry of $H(\Z_+)^{\oplus 2}$ which we denote by~$D'$.
 For the matrices $A$ and $B$ we are free to pick any integral lifts $A'$, $B'$; they play no role in the resulting integral matrix being an isometry.   Since $A$ is invertible and~$GL_{h-1}(\Z)\to GL_{h-1}(\Z_2)$ is surjective, we can assume that $A'$ is invertible.
	   It follows that the following matrix is also invertible:
	\[
\alpha_+:= \begin{pmatrix}
		A' & B' \\
		\mathbf{0} & D'
	\end{pmatrix}.
	\]
	Since $D'$ is an isometry of $H(\Z_+)$,  so is~$\alpha_+$.  \qedhere
\end{proof}	

\subsection{Stably realising $\lambda_b$}
\label{sub:StablyRealising}

Recall that $U_h \subset S^4$ denotes the unknotted surface with non-orientable genus~$h$ and normal Euler number~$-2h$, and that we set $(M_{U_h},\lambda_{U_h}):=(H_2(\widetilde{X}_{U_h}),\lambda_{X_{U_h}}).$
Combining the previous propositions with the work of Hambleton-Riehm yields the following result.
\begin{proposition}
\label{prop:alpha}
If~$b'\colon \Z^{h-1} \times \Z^{h-1} \to \Z$ is an even definite non-singular symmetric bilinear form and $b:=b'\oplus(1)$, then there exists a~$\Z[\Z_2]$-isometry
	\[
		\alpha\colon (M_{U_h},\lambda_{U_h})\oplus H(\Z[\Z_2])^{\oplus 2} \xrightarrow{\cong} (M,\lambda_b)\oplus H(\Z[\Z_2])^{\oplus 2}.
	\]
	In particular, the form $\lambda_b$ is stably realisable.
\end{proposition}	
\begin{proof}
Proposition~\ref{prop:alpha-} and Proposition~\ref{prop:alpha+} produce a compatible pair of isometries~$\alpha_-$ and~$\alpha_+$.
The existence of $\alpha$ now follows from Proposition~\ref{prop:pullback_isometry}.
The hermitian form~$\lambda_b$ is stably realisable because~$(M_{U_h},\lambda_{U_h})\oplus H(\Z[\Z_2])^{\oplus 2} $ is isometric to the $\Z[\Z_2]$-intersection form of $X_{U_h} \#_{2} S^2\times S^2$.
\end{proof}

	\subsection{Realising the form by a surface}
	\label{sub:ProofEnd}
	
	We now prove the previously mentioned refinement of our main theorem, which we restate for the reader.

\begin{customthm}
{\ref{thm:precise}}
If~$b'\colon \Z^{h-1} \times \Z^{h-1} \to \Z$ is an even definite non-singular symmetric bilinear form, then there exists a non-orientable~$\Z_2$-surface~$F\subset S^4$ such that~$Q_{\Sigma_2(F)} \cong b:=b'\oplus(1)$.
\end{customthm}
	\begin{proof}
	The output of Proposition~\ref{prop:alpha} is that $\lambda_b$ is stably realisable, and hence by applying Proposition~\ref{prop:StablyRealisable}, 	we obtain that $\lambda_b$ is realisable by a~$\Z_2$-manifold~$X'$ with boundary~$Y=\partial X_{U_h}$.
	Here, for the coefficient system we use the inclusion induced map~$\varphi \colon \pi_1(\partial X_{U_h}) \to \pi_1(X_{U_h}) \cong \Z_2$.
	
Since $Y=\partial X_{U_h}$ we then cap off~$X'$ with the normal bundle of $U_h$ in $S^4$ (which has Euler number $-2h$); denote the outcome by~$\widehat{X}'$.
	The choice of $\varphi$ ensures that~$\widehat{X}'$ is a closed simply-connected~$4$-manifold (this follows by a simple Seifert-Van Kampen argument).
	By construction, it contains a locally flat embedded $\Z_2$-surface, namely the $0$-section of the aforementioned disc bundle, which we denote by $F$.
	\begin{claim}
	\label{claim:ItsS4}
		The~$4$-manifold~$\widehat{X}'$ satisfies~$b_2(\widehat{X}')=0$.
	\end{claim}
	\begin{proof}[Proof of Claim~\ref{claim:ItsS4}]
		Since the~$4$-manifold~$X'$ has~$\pi_1(X')\cong \Z_2$ and~$\pi_1(Y) \to \pi_1(X') \cong \Z_2$ is surjective,  the argument from~\cite[Proof of Proposition 4.12]{ConwayOrsonPowell} shows that~$H_2(\widetilde{X'}) \cong \Z_- \oplus \Z[\Z_2]^{h-1}.$
		Since~$\widetilde{X'}$ double covers~$X'$,  the multiplicativity of the Euler characteristic gives~$b_2(X')=h-1$:
		$$
		2(1+b_2(X'))
		=2\chi(X')
		=\chi(\widetilde{X'})
		=1+b_2(\widetilde{X'})
		=2+2(h-1).
		$$
		The additivity of the Euler characteristic then shows that~$b_2(\widehat{X}')=0$:
		$$
		2+b_2(\widehat{X}')
		=\chi(\widehat{X}')
		=\chi(X')+\chi(F)
		=(1+b_2(X'))+(2-h)
		=h+(2-h)
		=2.
		$$
		This concludes the proof of Claim~\ref{claim:ItsS4}.
	\end{proof}
	Thus~$\widehat{X}'$ is a closed simply-connected with vanishing~$b_2$.
	It is therefore homeomorphic to~$S^4$ (by Freedman's work~\cite{Freedman}) and contains a~$\Z_2$-surface~$F$ with non-orientable genus $h$,  Euler number~$-2h$, and
exterior having equivariant intersection form isometric to~$\lambda_b$.

It remains to prove that~$Q_{\Sigma_2(F)}\cong b$.
Since~$\lambda_{X_F} \cong \lambda_b$ we deduce that~$(\lambda_{X_F})_- \cong 2Q_{\widetilde{X}_F}^{nd}$ is isometric to~$(\lambda_b)_- \cong 2b_{ext}^{nd}$.
Cancelling the~$2$s, which is possible since these forms are~$\Z$-valued, implies that~$b_{ext}^{nd} \cong Q_{\widetilde{X}_F}^{nd}$. 
 Thus $b$ and $Q_{\Sigma_2(F)}$ share $b_{ext}^{nd} \cong Q_{\widetilde{X}_F}^{nd}$ as an index $2$ subform.  
 Since~$h\not\equiv 4\mod{8}$, results well-known to lattice theorists---e.g.\ \cite[Proposition-Definition~3.2(i) and Proposition~3.8~(i)]{Chenevier} (see Proposition~\ref{prop:indextwo} for a translation of Chenevier's statements to our setting)---imply that $Q_{\Sigma_2(F)}\cong b$.\qedhere
\end{proof}

\appendix
\section{Neighbouring lattices}\label{sec:lattices}

The purpose of this section is to translate elements of the work of Chenevier~\cite{Chenevier} on Kneser neighbours to our setting for the purpose of applying them at the end of the proof of Theorem~\ref{thm:precise}.  This translation is purely formal.

\begin{definition}
	A \emph{lattice} of rank $h$ is a subgroup $L \subset \R^h$ generated by a basis of $\R^h$.  
	Two lattices~$L_1,L_2$ are \emph{isometric} if there is an isomorphism $\R^h\to \R^h$ preserving the standard inner product that sends $L_1$ to $L_2$.
\end{definition}

Given an unimodular lattice $L$,  we pick a generating basis $(v_i)$ and take the \emph{Gram matrix} $M_L$ whose entries are $v_i \cdot v_j$ where $\cdot$ denotes the standard inner product.  
This matrix is only well-defined up to conjugacy, but the underlying non-singular symmetric bilinear form is well-defined and we call it $b_L$.  
This procedure has an inverse.  
Given a positive definite non-singular symmetric bilinear form~$b$, use the standard basis of $\R^h$ to write~$b$ as a symmetric, positive definite matrix~$M$.  
For every such matrix, there exists a decomposition~$M=B^TB$ and the columns of $B$ can now be taken to be a new basis for $\R^h$ such that $M$ is the Gram matrix of~$\cdot$ with respect to this basis.
Taking the span of this basis gives a lattice $L_b$.  
It is not hard to see that $b_{L_b}=b$ and $L_{b_L}=L$.  Similarly, one can check that the concepts of isometry for lattices and forms coincide.

We recall the definition of a neighbouring lattice due to Kneser~\cite{Kneser}; see also e.g.~\cite{Voight} and~\cite[Section 3.1]{Chenevier}).

\begin{definition}
Given a unimodular lattice $L$ and an integer $d \geq 1$, a unimodular lattice~$N$ is a~\emph{$d$-neighbour} of~$L$ if~$L/(L\cap N)\cong \Z_d$, i.e.\ if~$N$ intersects~$L$ precisely in an index $d$ sublattice.
When $d=2$,  we refer to $N$ as a \emph{neighbour} of $L$.
\end{definition}

It is known that~$L$ is a $d$-neighbour of~$N$ if and only if~$N$ is a $d$-neighbour of~$L$; see e.g.~\cite[Section~3.1]{Chenevier}.
Restricting to the case $d=2$,  we discuss the exterior non-degenerate form in the setting of lattices and relate it to neighbours.

Given a unimodular lattice~$L$ and a $d$-primitive vector $x \in L$ (meaning that $[x]\in L/dL$ generates a subgroup of order $d$~\cite[Section 3.1]{Chenevier}), Chenevier defines 
\[
M_d(L;x):=\{m\in L\mid m\cdot x\equiv 0\mod{d}\}.
\] 
When $L$ is odd, characteristic vectors $\xi \in L$ are $2$-primitive, 
in which case $M_2(L;\xi)$ does not depend on the choice of the characteristic vector since
\[
M_2(L):=M_2(L;\xi)
=\{m\in L\mid m\cdot \xi\equiv 0\mod{2}\}
=\{m\in L\mid m\cdot m\equiv 0\mod{2}\}.
\]
This allows us to state the following.

\begin{lemma}\label{lem:formstolattices}
	Let~$b$ and~$c$ be rank~$h$ odd definite non-singular symmetric bilinear forms.  If there is an isometry~$b^{nd}_{ext}\cong c^{nd}_{ext}$,  then there exists an isometry $M_2(L_b)\cong M_2(L_c)$ and hence either $L_b$ is isometric to a neighbour of~$L_c$ sharing $M_2(L_c)$,  or~$L_b\cong L_c$.
\end{lemma}

\begin{proof}
	The first statement follows from the fact that $L_{b^{nd}_{ext}}=M_2(L_b)$ which is immediately clear from the definitions.  Let $\psi\colon \Z^h\to\Z^h$ denote the isometry.  Then~$\psi(L_b)\cap L_c$ contains~$M_2(L_c)$
	 and since the intersection must be a sublattice, either $\psi(L_b)\cap L_c=L_c$ or $\psi(L_b)\cap L_c=M_2(L_c)$. 
	In the latter case, then~$\psi(L_b)$ and~$L_c$ are neighbours since~$M_2(L_c)$ is index two in~$L_c$. 
	In the former case then~$\psi(L_b)=L_c$ and so $\psi$ was actually an isometry.
\end{proof}

Whereas the number of~$d$-neighbours of a lattice~$L$ may be large~\cite[Corollary 3.6]{Chenevier},  the number of~$d$-neighbours~$N$ with~$N \cap L=M_d(L;x)$ for a fixed~$d$-primitive vector~$x \in L$ is typically much smaller~\cite[Proposition 3.2]{Chenevier}.
We will state a particular case of this result in the case where~$d=2$ (so that~$e=2$ in Chenevier's notation).
In this setting,  the answer depends on~$x \cdot x$ mod~$8$ but since we are only interested in the case where~$x=\xi$ is characteristic,  we will formulate the results using the rank of the lattice instead: recall that for unimodular lattices,~$h\equiv \xi\cdot \xi \mod 8$; see e.g.~\cite[Lemma~1.2.20]{GompfStipsicz}.

We now state the main result we want to use which is a combination of \cite[Proposition-Definition 3.2 (i)]{Chenevier} and \cite[Proposition 3.8 (i)]{Chenevier}.

\begin{proposition}\label{prop:chenevier}
Let~$L$ be an odd integral unimodular lattice of rank~$h$.  
	\begin{enumerate}
\item If~$h\not\equiv 0,4 \mod 8$,  then~$L$ has no neighbours sharing the index two sublattice~$M_2(L)$.
		\item If~$h\equiv 0\mod{8}$, then~$L$ has precisely two neighbours sharing the index two sublattice~$M_2(L)$, both of which are even.
		\item If~$h\equiv 4\mod{8}$, then~$L$ has precisely two neighbours sharing the index two sublattice~$M_2(L)$, both of which are odd.
	\end{enumerate}
\end{proposition}
\begin{proof}
Consider \cite[Proposition-Definition~3.2 (i)]{Chenevier} with $d=2=e$ and~$x=\xi$ characteristic so that $M=M_2(L)=M_2(L;\xi)$.
The first item of this result states the number of neighbours $N$ with $L \cap N=M_2(L)$ is $2$ if $\xi \cdot \xi=0$ mod $4$ and $0$ otherwise.

This directly implies the first item of the proposition, so we assume that $h \equiv 0 \mod 4$ and focus on the next two items.
As mentioned above, in these cases $L$ admits two neigbhours $N$ with with $L \cap N=M_2(L)$.
In~\cite[Remark 3.3]{Chenevier}, Chenevier describes these neighbours explicitly and denotes them by $N_2(L;\xi,0)$ and $N_2(L;\xi,1)$.
Chenevier then shows that these neighbours are even if and only if $\xi\cdot\xi\equiv 0\mod{8}$~\cite[Proposition 3.8 (i)]{Chenevier}.
This directly implies the last two items of the proposition.
\end{proof}

We give the translation to forms in the guise that we utilised in Section~\ref{sec:Proof}.

\begin{proposition}\label{prop:indextwo}
	Let~$b,c$ be odd definite non-singular symmetric bilinear forms of rank~$h$ and assume~$b^{nd}_{ext}\cong c^{nd}_{ext}$.  
	If~$h\not\equiv 4\mod{8}$ then~$b\cong c$.
\end{proposition}

\begin{proof}

	Apply Lemma~\ref{lem:formstolattices} to deduce that~$L_b$ is isometric to a lattice~$L$ such that either $L=L_c$,  or~$L$ and~$L_c$ are neighbours sharing the index two sublattice~$M_2(L_c)$.  In the former case $L_b \cong L_c$,  so $b \cong c$, as required.  We therefore assume that $L$ and $L_c$ are neighbours over~$M_2(L_c)$.

	If~$h\not\equiv 0,4\mod{8}$, then by (1) of Proposition~\ref{prop:chenevier} we have that~$L$ and~$L_c$ cannot be neighbours over~$M_2(L_c)$, hence~$L=L_c$. 
	We again deduce that~$b\cong c$.
	
	If~$h\equiv 0\mod{8}$, then by (2) of Proposition~\ref{prop:chenevier} we have that all neighbours of~$L$ sharing the sublattice~$M_2(L_c)$ are even, and hence~$L$ and~$L_c$ are not neighbours. 
	 This gives~$L=L_c$ and so we deduce that~$b\cong c$. 
\end{proof}

\bibliographystyle{alpha}
\bibliography{BiblioRealisation}

\newcommand{\etalchar}[1]{$^{#1}$}
\begin{thebibliography}{BKK{\etalchar{+}}21}

\bibitem[BCD{\etalchar{+}}21]{Matrix}
Jonathan Bowden, Diarmuid Crowley, Jim Davis, Stefan Friedl, Carmen Rovi, and
  Stephan Tillmann.
\newblock Open problems in the topology of manifolds.
\newblock In {\em 2019--20 MATRIX annals}, pages 647--659. Cham: Springer,
  2021.

\bibitem[BKK{\etalchar{+}}21]{DET}
Stefan Behrens, Boldizs\'ar Kalm\'ar, Min~Hoon Kim, Mark Powell, and Arunima
  Ray, editors.
\newblock {\em The disc embedding theorem}.
\newblock Oxford University Press, Oxford, 2021.

\bibitem[BKR26]{K3}
{\.I}nan{\c{c}} Baykur, Robion Kirby, and Daniel Ruberman, editors.
\newblock {\em {K3}: a new problem list in low-dimensional topology (to
  appear)}.
\newblock Providence, RI: American Mathematical Society (AMS), 2026.

\bibitem[Che25]{Chenevier}
Ga\"etan Chenevier.
\newblock Unimodular hunting.
\newblock {\em Algebr. Geom.}, 12(6):769--812, 2025.

\bibitem[COP23]{ConwayOrsonPowell}
Anthony Conway, Patrick Orson, and Mark Powell.
\newblock Unknotting nonorientable surfaces.
\newblock 2023.
\newblock \url{https://arxiv.org/abs/2306.12305}.

\bibitem[CP23]{ConwayPowell}
Anthony Conway and Mark Powell.
\newblock Embedded surfaces with infinite cyclic knot group.
\newblock {\em Geom. Topol.}, 27(2):739--821, 2023.

\bibitem[CS99]{ConwaySloane}
J.~H. Conway and N.~J.~A. Sloane.
\newblock {\em Sphere packings, lattices and groups}, volume 290 of {\em
  Grundlehren der mathematischen Wissenschaften [Fundamental Principles of
  Mathematical Sciences]}.
\newblock Springer-Verlag, New York, third edition, 1999.
\newblock With additional contributions by E. Bannai, R. E. Borcherds, J.
  Leech, S. P. Norton, A. M. Odlyzko, R. A. Parker, L. Queen and B. B. Venkov.

\bibitem[Don83]{Donaldson}
Simon Donaldson.
\newblock An application of gauge theory to four dimensional topology.
\newblock {\em J. Differ. Geom.}, 18:279--315, 1983.

\bibitem[Edm89]{EdmondsAspects}
Allan Edmonds.
\newblock Aspects of group actions on four-manifolds.
\newblock {\em Topology Appl.}, 31(2):109--124, 1989.

\bibitem[Eic74]{Eichler}
Martin Eichler.
\newblock {\em Quadratische {F}ormen und orthogonale {G}ruppen}.
\newblock Die Grundlehren der mathematischen Wissenschaften, Band 63.
  Springer-Verlag, Berlin-New York, 1974.
\newblock Zweite Auflage.

\bibitem[Fin02]{Finashin}
Sergey Finashin.
\newblock Knotting of algebraic curves in {$\mathbb{C} P^2$}.
\newblock {\em Topology}, 41(1):47--55, 2002.

\bibitem[FKV88]{FinashinKreckViro}
Sergey Finashin, Matthias Kreck, and Oleg Viro.
\newblock Nondiffeomorphic but homeomorphic knottings of surfaces in the
  {$4$}-sphere.
\newblock In {\em Topology and geometry---{R}ohlin {S}eminar}, volume 1346 of
  {\em Lecture Notes in Math.}, pages 157--198. Springer, Berlin, 1988.

\bibitem[FQ90]{FreedmanQuinn}
Michael Freedman and Frank Quinn.
\newblock {\em Topology of 4-manifolds}, volume~39 of {\em Princeton
  Mathematical Series}.
\newblock Princeton University Press, Princeton, NJ, 1990.

\bibitem[Fre82]{Freedman}
Michael Freedman.
\newblock The topology of four-dimensional manifolds.
\newblock {\em J. Differential Geometry}, 17(3):357--453, 1982.

\bibitem[GHS13]{Gritsenko}
V.~Gritsenko, K.~Hulek, and G.~K. Sankaran.
\newblock Moduli of {K}3 surfaces and irreducible symplectic manifolds.
\newblock In {\em Handbook of moduli. {V}ol. {I}}, volume~24 of {\em Adv. Lect.
  Math. (ALM)}, pages 459--526. Int. Press, Somerville, MA, 2013.

\bibitem[GS99]{GompfStipsicz}
Robert Gompf and Andr\'{a}s Stipsicz.
\newblock {\em {$4$}-manifolds and {K}irby calculus}, volume~20 of {\em
  Graduate Studies in Mathematics}.
\newblock American Mathematical Society, Providence, RI, 1999.

\bibitem[HK88]{HambletonKreck}
Ian Hambleton and Matthias Kreck.
\newblock On the classification of topological {$4$}-manifolds with finite
  fundamental group.
\newblock {\em Math. Ann.}, 280(1):85--104, 1988.

\bibitem[HR78]{HambletonRiehm}
Ian Hambleton and Carl Riehm.
\newblock Splitting of {H}ermitian forms over group rings.
\newblock {\em Invent. Math.}, 45(1):19--33, 1978.

\bibitem[Kam02]{KamadaBook}
Seiichi Kamada.
\newblock {\em Braid and knot theory in dimension four}, volume~95 of {\em
  Mathematical Surveys and Monographs}.
\newblock American Mathematical Society, Providence, RI, 2002.

\bibitem[Kaw96]{KawauchiSurvey}
Akio Kawauchi.
\newblock {\em A survey of knot theory}.
\newblock Birkh\"{a}user Verlag, Basel, 1996.
\newblock Translated and revised from the 1990 Japanese original by the author.

\bibitem[KL90]{KleidmanLiebeck}
Peter Kleidman and Martin Liebeck.
\newblock {\em The subgroup structure of the finite classical groups}, volume
  129 of {\em London Mathematical Society Lecture Note Series}.
\newblock Cambridge University Press, Cambridge, 1990.

\bibitem[Kne57]{Kneser}
Martin Kneser.
\newblock Klassenzahlen definiter quadratischer {F}ormen.
\newblock {\em Arch. Math.}, 8:241--250, 1957.

\bibitem[Kre90]{KreckOnTheHomeomorphism}
Matthias Kreck.
\newblock On the homeomorphism classification of smooth knotted surfaces in the
  4- sphere.
\newblock Geometry of low-dimensional manifolds. 1: {Gauge} theory and
  algebraic surfaces, {Proc}. {Symp}., {Durham}/{UK} 1989, {Lond}. {Math}.
  {Soc}. {Lect}. {Note} {Ser}. 150, 63-72 (1990)., 1990.

\bibitem[KSTY99]{KatanagaSaekiTeragaitoYamada}
Atsuko Katanaga, Osamu Saeki, Masakazu Teragaito, and Yuichi Yamada.
\newblock Gluck surgery along a {$2$}-sphere in a {$4$}-manifold is realized by
  surgery along a projective plane.
\newblock {\em Michigan Math. J.}, 46(3):555--571, 1999.

\bibitem[Kug84]{Kuga}
Kenichi Kuga.
\newblock Representing homology classes of {$S\sp{2}\times S\sp{2}$}.
\newblock {\em Topology}, 23(2):133--137, 1984.

\bibitem[KV86]{KwasikVogel}
Slawomir Kwasik and Pierre Vogel.
\newblock Asymmetric four-dimensional manifolds.
\newblock {\em Duke Math. J.}, 53(3):759--764, 1986.

\bibitem[Law84]{Lawson}
Terry Lawson.
\newblock Detecting the standard embedding of {{\({\mathbb{R}}P^ 2\)}} in
  {{\(S^ 4\)}}.
\newblock {\em Math. Ann.}, 267:439--448, 1984.

\bibitem[Luo88]{Luo}
Feng Luo.
\newblock Representing homology classes of {${\bf C}{\rm P}^2\#\;\overline{{\bf
  C}{\rm P}}{}^2$}.
\newblock {\em Pacific J. Math.}, 133(1):137--140, 1988.

\bibitem[LW97]{LeeWilczyGenus}
Ronnie Lee and Dariusz~M. Wilczy\'{n}ski.
\newblock Representing homology classes by locally flat surfaces of minimum
  genus.
\newblock {\em Amer. J. Math.}, 119(5):1119--1137, 1997.

\bibitem[Miy23]{Miyazawa}
Jin Miyazawa.
\newblock A gauge theoretic invariant of embedded surfaces in $4$-manifolds and
  exotic ${P}^2$-knots.
\newblock 2023.
\newblock \url{https://arxiv.org/abs/2312.02041}.

\bibitem[MOJ{\etalchar{+}}23]{MaticOzturkReyesStipsiczUrzua}
Gordana Mati\'c, Ferit \"Ozt\"urk, Reyes Javier, Andr\'as Stipsicz, and Urz\'ua
  Giancarlo.
\newblock An exotic $5 \mathbb{R} {P}^2$ in the $4$-sphere.
\newblock 2023.
\newblock \url{https://arxiv.org/abs/2312.03617 }.

\bibitem[Nik79]{Nikulin}
Viacheslav Nikulin.
\newblock Integer symmetric bilinear forms and some of their geometric
  applications.
\newblock {\em Izv. Akad. Nauk SSSR Ser. Mat.}, 43(1):111--177, 238, 1979.

\bibitem[Pen24]{Pencovitch}
Mark Pencovitch.
\newblock Unknotting nonorientable surfaces of genus 4 and 5.
\newblock {\em Linear Algebra Appl.}, 702:195--217, 2024.

\bibitem[PRT25]{PowellRayTeichner}
Mark Powell, Arunima Ray, and Peter Teichner.
\newblock The 4-dimensional disc embedding theorem and dual spheres.
\newblock {\em Selecta Math. (N.S.)}, 31(4):Paper No. 80, 25, 2025.

\bibitem[Rei57]{Reiner}
Irving Reiner.
\newblock Integral representations of cyclic groups of prime order.
\newblock {\em Proc. Am. Math. Soc.}, 8:142--146, 1957.

\bibitem[Rud84]{Rudolph}
Lee Rudolph.
\newblock Some topologically locally-flat surfaces in the complex projective
  plane.
\newblock {\em Comment. Math. Helv.}, 59(4):592--599, 1984.

\bibitem[Ser78]{Serre}
Jean-Pierre Serre.
\newblock {\em A course in arithmetic. {Translation} of ''{Cours}
  d'arithmetique''. 2nd corr. print}, volume~7 of {\em Grad. Texts Math.}
\newblock Springer, Cham, 1978.

\bibitem[Tor25]{Torres}
Rafael Torres.
\newblock Smoothly knotted and topologically unknotted nullhomologous surfaces
  in 4-manifolds.
\newblock {\em Ann. Inst. Fourier (Grenoble)}, 75(6):2501--2527, 2025.

\bibitem[Voi23]{Voight}
John Voight.
\newblock Kneser's method of neighbors.
\newblock {\em Arch. Math. (Basel)}, 121(5-6):537--557, 2023.

\end{thebibliography}

\end{document}